\documentclass[12pt]{article}

\usepackage{amssymb,amsthm,amsmath}
\usepackage{colortbl}
\usepackage[pdftex]{graphicx}
\usepackage[letterpaper,hmargin=1in,vmargin=1in]{geometry}
\usepackage{url}

\usepackage{longtable}

\usepackage[titletoc,toc,title]{appendix}

\definecolor{teal}{RGB}{0,128,128}
\definecolor{darkpurple}{RGB}{128,0,128}

\newcommand{\comment}[1]{}

\renewcommand{\thefootnote}{\fnsymbol{footnote}}

\def \cA {{\cal A}}
\def \cB {{\cal B}}

\def \cG {{\cal G}}
\def \cH {{\cal H}}

\def \Z {\mathbb Z}
\def \cR {{\cal R}}
\def \cS {{\cal S}}

\newtheorem{theorem}{Theorem}[section]
\newtheorem{lemma}[theorem]{Lemma}
\newtheorem{corollary}[theorem]{Corollary}

\theoremstyle{definition}
\newtheorem{definition}[theorem]{Definition}
\newtheorem{example}[theorem]{Example}

\newcommand{\ghd}{\mathrm{GHD}}
\newcommand{\ghf}{\mathrm{GHF}}
\newcommand{\KS}{\mathrm{KS}}
\newcommand{\soma}{\mathrm{SOMA}}
\newcommand{\bibd}{\mathrm{BIBD}}
\newcommand{\rtd}{\mathrm{RTD}}
\newcommand{\pbd}{\mathrm{PBD}}

\newcommand{\mols}{\mathrm{MOLS}}

\newcommand{\imols}{\mathrm{IMOLS}}

\newcommand{\XX}{\mathbf{X}}
\newcommand{\vv}{\mathbf{v}}
\newcommand{\kk}{\mathbf{k}}
\renewcommand{\tt}{\mathbf{t}}

\title{%\vspace{-1.5cm} 
On generalized Howell designs with block size three}

\author{
R.~Julian~R.~Abel\footnotemark[1]
\and
Robert F.~Bailey\footnotemark[2] \footnotemark[5]
\and
Andrea C.~Burgess\footnotemark[3] \footnotemark[6]  
\and
Peter Danziger\footnotemark[4] \footnotemark[6]
\and
Eric Mendelsohn\footnotemark[4] 
}

\begin{document}
\maketitle

\footnotetext[1]{School of Mathematics and Statistics, University of New South Wales, Sydney, NSW 2052, Australia.  Email: \texttt{r.j.abel@unsw.edu.au}}
\footnotetext[2]{Division of Science (Mathematics), Grenfell Campus, Memorial University of Newfoundland, Corner Brook, NL  A2H 6P9, Canada.  Email: \texttt{rbailey@grenfell.mun.ca}}
\footnotetext[3]{Department of Mathematical Sciences, University of New Brunswick, 100 Tucker Park Rd., Saint John, NB  E2L 4L5, Canada.  Email: \texttt{andrea.burgess@unb.ca}}
\footnotetext[4]{Department of Mathematics, Ryerson University, 350 Victoria St., Toronto, ON  M5B 2K3, Canada.  Email: \texttt{danziger@ryerson.ca, emendelso@gmail.com}}
\footnotetext[5]{Supported by Vice-President (Grenfell Campus) Research Fund, Memorial University of Newfoundland.}
\footnotetext[6]{Supported by an NSERC Discovery Grant.}

\renewcommand{\thefootnote}{\arabic{footnote}}

\vspace{-.5cm}
\begin{abstract}
In this paper, we examine a class of doubly resolvable combinatorial objects. Let $t, k, \lambda, s$ and $v$ be nonnegative integers, and let $X$ be a set of $v$ symbols.  A {\em generalized Howell design}, denoted $t$-$\ghd_{k}(s,v;\lambda)$, is an $s\times s$ array, each cell of which is either empty or contains a $k$-set of symbols from $X$, called a {\em block}, such that:
(i) each symbol appears exactly once in each row and in each column (i.e.\ each row and column is a resolution of $X$);
(ii) no $t$-subset of elements from $X$ appears in more than $\lambda$ cells.
Particular instances of the parameters correspond to Howell designs, doubly resolvable balanced incomplete block designs (including Kirkman squares), doubly resolvable nearly Kirkman triple systems, and simple orthogonal multi-arrays (which themselves generalize mutually orthogonal Latin squares).  Generalized Howell designs also have connections with permutation arrays and multiply constant-weight codes.

In this paper, we concentrate on the case that $t=2$, $k=3$ and $\lambda=1$, and write $\ghd(s,v)$.  In this case, the number of empty cells in each row and column falls between 0 and $(s-1)/3$.  
Previous work has considered the existence of GHDs on either end of the spectrum, with at most 1 or at least $(s-2)/3$ empty cells in each row or column.
In the case of one empty cell, we correct some results of Wang and Du, and show that there exists a $\ghd(n+1,3n)$ if and only if $n \geq 6$, except possibly for  $n=6$. 
In the case of two empty cells, we show that there exists a $\ghd(n+2,3n)$ if and only if $n \geq 6$.  
Noting that the proportion of cells in a given row or column of a $\ghd(s,v)$ which are empty falls in the interval $[0,1/3)$, we prove that for any $\pi \in [0,5/18]$, there 
is a $\ghd(s,v)$ whose proportion of empty cells in a row or column is arbitrarily close to $\pi$.  
\end{abstract}

{\bf Keywords:} Generalized Howell designs, triple systems, doubly resolvable designs. 

{\bf MSC2010:} Primary: 05B07, 05B15; Secondary: 05B40, 94B25

%\newpage

\section{Introduction}

Combinatorial designs on square arrays have been the subject of much attention, with mutually orthogonal Latin squares being the most natural example.  Block designs with two orthogonal resolutions can also be thought of in this way, with the rows and columns of the array labelled by the resolution classes.  In this paper, we consider generalized Howell designs, which are objects that in some sense fill in the gap between these two cases.
We refer the reader to \cite{Handbook} for background on these objects and design theory in general.

\subsection{Definition and examples} \label{DefnSection}

In this paper, we examine a class of doubly resolvable designs, defined below, which generalize a number of well-known objects.  

\begin{definition}
Let $t, k, \lambda, s$ and $v$ be nonnegative integers, and let $X$ be a set of $v$ symbols.  A {\em generalized Howell design}, denoted $t$-$\ghd_{k}(s,v;\lambda)$, is an $s\times s$ array, each cell of which is either empty or contains a $k$-set of symbols from $X$, called a {\em block}, such that:
\begin{enumerate}
\item
\label{latin}
each symbol appears exactly once in each row and in each column (i.e.\ each row and column is a resolution of $X$);
\item
no $t$-subset of elements from $X$ appears in more than $\lambda$ cells.
\end{enumerate}
\end{definition}
In the case that $t=2$, $k=2$ and $\lambda=1$, a $2$-$\ghd_2(s,v;1)$ is known as a {\em Howell design} $\mathrm{H}(s,v)$.  In the literature, two different generalizations 
of Howell designs have been proposed, both of which can be incorporated into the definition above: these are due to Deza and Vanstone~\cite{DezaVanstone} (which 
corresponds to the case that $t=2$) and to Rosa~\cite{Rosa} (which corresponds to the case that $t=k$).  The objects in question have appeared under several names 
in the literature.  The term {\em generalized Howell design} appears in both~\cite{DezaVanstone} and~\cite{Rosa} in reference to the particular cases studied in these 
papers, and also in papers such as~\cite{WangDu}.  The term {\em generalized Kirkman square} or GKS  has more recently been introduced to the literature (see~\cite{DuAbelWang,Etzion}); 
in particular, in these papers, a $\mathrm{GKS}_k(v;1,\lambda;s)$ is defined to be what we have defined as a $2$-$\ghd_k(s,v;\lambda)$.  
In this paper we will concentrate on the case when $t=2$, $k=3$ and $\lambda=1$, in which case we omit these parameters in the notation and simply write $\ghd(s,v)$, or $\ghd_k(s,v)$ for more general $k$.

Two obvious necessary conditions for the existence of a non-trivial $2$-$\ghd_k(s,v;\lambda)$ are that $v\equiv 0 \pmod k$ and that $\frac{v}{k}\leq s \leq \frac{\lambda(v-1)}{k-1}$. In particular, when $k=3$, $t=2$ and $\lambda=1$, we have that $\frac{v}{3}\leq s \leq \frac{v-1}{2}$. Since a $t$-$\ghd_k(s,v;\lambda)$ contains exactly $n=\frac{v}{k}$ non-empty cells in each row and column, it can be helpful to write $t$-$\ghd_k(n+e, kn;\lambda)$ (or $\ghd(n+e, 3n)$ in the case $k=3$, $t=2$ and $\lambda=1$), where $e$ is then the number of empty cells in each row and column.

A Howell design  with $v=s+1$, i.e. an $\mathrm{H}(s,s+1)$ is called a {\em Room square}. The study of Room squares goes back to the original work of Kirkman \cite{Kirkman} in 1850, where 
he presents a Room square with side length~$7$, i.e.\ an $\mathrm{H}(7,8)$.  The name of this object, however, is in reference to T.~G.~Room~\cite{Room}, who also constructed  an $\mathrm{H}(7,8)$, and in 
addition showed that there is no $\mathrm{H}(3,4)$ or $\mathrm{H}(5,6)$.  The existence of Room squares was settled in 1975 by Mullin and Wallis \cite{Mullin}; for a survey on Room squares see \cite{blue book}. 
\begin{theorem}[Mullin and Wallis \cite{Mullin}]
\label{Room squares}
There exists a Room square of side $s$ if and only if $s$ is odd and either $s=1$ or $s \geq 7$.
\end{theorem}

More generally, Stinson \cite{Stinson} showed existence of Howell designs with odd side $s$ and Anderson, Schellenberg and Stinson \cite{ASS} showed existence of Howell designs with even side $s$. 
We thus have the following.
\begin{theorem}[Stinson \cite{Stinson}; Anderson, Schellenberg and Stinson \cite{ASS}]
\label{HD}
There exists an $\mathrm{H}(s,2n)$ if and only if $n=0$ or $n \leq s \leq 2n-1$, $(s,2n)\not\in \{(2,4), (3,4), (5,6), (5,8)\}$.
\end{theorem}

A $2$-$\ghd_k(s,v;\lambda)$ with $s=\frac{\lambda(v-1)}{k-1}$ is equivalent to a {\em doubly resolvable balanced incomplete block design} $\bibd(v,k,\lambda)$.  Doubly resolvable designs and related objects have been studied, for example, in~\cite{OCD, Curran Vanstone, Fuji-Hara Vanstone, Fuji-Hara Vanstone TD, Vanstone AG, Lamken 09, Vanstone 80, Vanstone 82}.  In particular, Fuji-Hara and Vanstone investigated orthogonal resolutions in affine geometries, showing the existence of a doubly resolvable $\bibd(q^n,q,1)$ for prime powers $q$ and integers $n \geq 3$.  Asymptotic existence of doubly resolvable $\bibd(v,k,1)$ was shown by Rosa and Vanstone~\cite{Rosa Vanstone} for $k=3$ and by Lamken~\cite{Lamken 09} for general $k$.  

For $t=2$, $k=3$ and $\lambda=1$, a $\ghd(s,v)$ with $s = \lfloor\frac{v-1}{2}\rfloor$ corresponds to a {\em Kirkman square}, $\KS(v)$, (i.e.\ a doubly resolvable Kirkman triple system of order $v$) when $v \equiv 3 \pmod 6$ and a {\em doubly resolvable nearly Kirkman triple system}, $\mathrm{DRNKTS}(v)$, when $v\equiv 0 \pmod 6$.  A Kirkman square of order $3$ is trivial to construct.  Mathon and Vanstone~\cite{Mathon Vanstone} showed the non-existence of a $\KS(9)$ or $\KS(15)$, while the non-existence of a $\mathrm{DRNKTS}(6)$ and $\mathrm{DRNKTS}(12)$ follows from Kotzig and Rosa~\cite{KotzigRosa74}.  
For many years, the smallest known example of a $\ghd(s,v)$ with $s = \lfloor\frac{v-1}{2}\rfloor$ (other than the trivial case of $s=1$, $v=3$) was a $\mathrm{DRNKTS}(24)$, 
found by Smith in 1977~\cite{Smith77}.  However, the smallest possible example of such a GHD 
with $v\equiv 0 \pmod 6$    
is for $v=18$ and $s=8$; these were recently obtained and classified up to isomorphism in~\cite{Finland}, from which 
the following example is taken.

\begin{example} \label{FinlandExample}
The following is a $\ghd(8,18)$, or equivalently a $\mathrm{DRNKTS}(18)$.
\renewcommand{\arraystretch}{1.25}
\[
\begin{array}{|c|c|c|c|c|c|c|c|} \hline
  	&	    & agh & bkm & ejp & fir & clq & dno \\ \hline
	  &  	  & bln & aij & cho & dgq & ekr & fmp \\ \hline
acd & boq &		  &	    & gmr & hkp & fjn & eil \\ \hline
bpr & aef &	  	&	    & inq & jlo & dhm & cgk \\ \hline
fgl & dik & emq & cnr &		  &  	  & aop & bhj \\ \hline
ehn & cjm & fko & dlp &		  & 	  & bgi & aqr \\ \hline
imo & gnp & djr & fhq & akl & bce &		  & \\ \hline
jkq & hlr & cip & ego & bdf & amn & 	  & \\ \hline
\end{array}
\]
\renewcommand{\arraystretch}{1.0}
\end{example}

The existence of Kirkman squares was settled by Colbourn, Lamken, Ling and Mills in \cite{CLLM}, with 23 possible exceptions, 11 of which where solved in \cite{ALW} and one ($v=351$) was solved in~\cite{ACCLWW}.
Abel, Chan, Colbourn, Lamken, Wang and Wang \cite{ACCLWW} have determined the existence of doubly resolvable nearly Kirkman triple systems, with 64 possible exceptions. We thus have the following result.
\begin{theorem}[\cite{ACCLWW,ALW,CLLM}]
\label{Kirkman squares}
Let $v \equiv 0 \pmod{3}$ be a positive integer.  Then a $\ghd(\lfloor\frac{v-1}{2}\rfloor,v)$ exists if and only if 
either $v=3$ or $v\geq 18$, with 75 possible exceptions.  The symbol $N$ will be used to denote this set of possible 
exceptions throughout the paper.
\end{theorem}

Recently, Du, Abel and Wang~\cite{DuAbelWang} have proved the existence of $\ghd(\frac{v-4}{2},v)$ for $v \equiv 0 \pmod{12}$ with at most 15 possible exceptions and $\ghd(\frac{v-6}{2},v)$ for $v \equiv 0 \pmod{6}$ and $v>18$ with at most 31 possible exceptions.  Two of these possible exceptions ($\ghd(9,24)$ and $\ghd(12,30)$) will be addressed later in this paper.

At the other end of the spectrum, the case when $s=n$, $v=kn$, $t=2$ and $\lambda=1$ is known as a $\soma(k,n)$.  SOMAs, or {\em simple orthogonal multi-arrays}, were introduced by Phillips and Wallis~\cite{PhillipsWallis} and have been investigated by Soicher~\cite{Soicher} and Arhin~\cite{ArhinThesis,Arhin}. We note that the existence of a $2$-$\ghd_{k}(n,kn;1)$ is guaranteed by the existence of $k$ mutually orthogonal Latin squares (MOLS) of side $n$. (For more information on Latin squares and MOLS, see~\cite[Part III]{Handbook}.)  It is well known that there exist 3 MOLS of side $n$ for every $n\neq 2,3,6,10$. Interestingly, even though the corresponding set of 3 MOLS is not known to exist (and is known not to for $n=2,3,6$), the existence of a $\ghd(6, 18)$ and $\ghd(10, 30)$ has been shown by Brickell~\cite{brickell} and Soicher~\cite{Soicher}, respectively. A $\ghd(1,3)$ exists trivially but it is easily seen that the existence of a $\ghd(2,6)$ or $\ghd(3,9)$ is impossible.  
\begin{theorem} \label{618}
\label{Latin}
There exists a $\ghd(n,3n)$ if and only if $n =1$ or $n \geq 4$.
\end{theorem}

\subsection{GHDs, permutation arrays and codes}
\label{PA-code}

In~\cite{DezaVanstone}, Deza and Vanstone noted $2$-$\ghd_k(s,v;\lambda)$ are equivalent to a particular type of permutation array.
\begin{definition}
A {\em permutation array} $\mathrm{PA}(s,\lambda;v)$ on a set $S$ of $s$ symbols is an $v \times s$ array such that each row is a permutation of $S$ and any two rows agree in at most $\lambda$ columns.  A permutation array in which each pair of rows agree in exactly $\lambda$ symbols is called {\em equidistant}.  If, in any given column, each symbol appears either $0$ or $k$ times, then a permutation array is said to be {\em $k$-uniform}.  
\end{definition}
The rows of a permutation array form an error-correcting code with minimum distance $s-\lambda$.  Codes formed in this manner, called {\em permutation codes}, have more recently attracted attention due to their applications to powerline communications; see~\cite{ChuColbournDukes, ColbournKloveLing, Huczynska}.  In~\cite{DezaVanstone}, it is noted that a $2$-$\ghd_k(s,v;\lambda)$ is equivalent to a $k$-uniform $\mathrm{PA}(s,\lambda;v)$.  In the case that $s=\frac{\lambda(v-1)}{k-1}$, so that the $\ghd$ is a doubly resolvable BIBD, the permutation array is equidistant.  

There are also known connections between GHDs and other classes of codes.  A {\em doubly constant weight code} is a binary constant-weight code with length $n=n_1+n_2$ and weight $w=w_1+w_2$, where each codeword has $w_1$ 1s in the first $n_1$ positions and $w_2$ 1s in the final $n_2$ positions.  Such a code with minimum distance $d$ is referred to as a $(w_1,n_1,w_2,n_2,d)$-code.  In~\cite{Etzion}, Etzion describes connections between certain classes of designs and optimal doubly constant weight codes.  In particular, a $2$-$\ghd_k(s,v;1)$ gives a $(2,2s,k,v,2k+2)$-code.

\subsection{GHDs as generalized packing designs}

The notion of {\em generalized $t$-designs} introduced by Cameron in 2009 \cite{Cameron09}, and the broader notion of {\em generalized packing designs} introduced by Bailey and Burgess \cite{packings}, provide a common framework for studying many classes of combinatorial designs.  As described below, generalized Howell designs fit neatly into this framework.  Recently, Chee, Kiah, Zhang and Zhang~\cite{CKZZ} noted that generalized packing designs are equivalent to {\em multiply constant-weight codes}, introduced in~\cite{Chee2}, which themselves generalize doubly constant-weight codes as discussed in Section~\ref{PA-code}.  

Suppose that $v,k,t,\lambda,m$ are integers where $v \geq k \geq t \geq 1$, $\lambda \geq 1$ and $m\geq 1$.  Let $\vv = (v_1,v_2,\ldots,v_m)$ and $\kk=(k_1,k_2,\ldots,k_m)$ be $m$-tuples 
of positive integers with sum $v$ and $k$ respectively, where $k_i\leq v_i$ for all $i$, and let $\XX = (X_1,X_2,\ldots,X_m)$ be an $m$-tuple of pairwise disjoint sets, where $|X_i|=v_i$.  
We say that an $m$-tuple of {\em non-negative} integers $\tt=(t_1,t_2,\ldots,t_m)$ is {\em admissible} if $t_i\leq k_i$ for all $i$ and the entries $t_i$ sum to $t$.
Now we define a {\em $t$-$(\vv,\kk,\lambda)$ generalized packing design} to be a collection $\mathcal{P}$ of $m$-tuples of sets $(B_1,B_2,\ldots,B_m)$, where $B_i\subseteq X_i$ and $|B_i|=k_i$ for all $i$, 
with the property that for any admissible $\tt$, any $m$-tuple of sets $(T_1,T_2,\ldots,T_m)$ (where $T_i\subseteq X_i$ and $|T_i|=t_i$ for all $i$) is contained in at 
most $\lambda$ members of $\mathcal{P}$.  We say a generalized packing is {\em optimal} if it contains the largest possible number of blocks. (See~\cite{packings} for further details, and for numerous examples.)

The connection with GHDs is the following: any $2$-$\ghd_k(s,n)$ forms an optimal $2$-$(\vv,\kk,1)$ generalized packing, where $\vv=(n,s,s)$ and $\kk=(k,1,1)$.  The ``point set'' of the 
generalized packing is formed from the points of the GHD, together with the row labels and the column labels.  Since $t=2$, the only possible admissible 
triples $\tt$ are $(2,0,0)$, $(1,0,1)$, $(0,1,0)$ and $(0,1,1)$.  The first of these tells us that no pair of points may occur in more than one block of the GHD; the second and third tell us 
that no point may be repeated in any row or any column; the last tells us that any entry of the $s\times s$ array may contain only one block of the GHD.

In fact, GHDs may be used to obtain $2$-$(\vv,\kk,1)$ generalized packings for $\kk=(k,1,1)$ more generally, for more arbitrary $\vv$.  If $\vv=(n,r,s)$, we may obtain a rectangular array 
by (for instance) deleting rows or adding empty columns.  Depending on the value of $n$ relative to $r,s$, we may need to delete points and/or blocks.  This idea is discussed 
further in~\cite[Section 3.5]{packings}.

\section{Terminology}

\subsection{Designs and resolutions}

In this section, we discuss various useful classes of combinatorial designs.  For more information on these and related objects, see~\cite{Handbook}.

We say that $(X, \cB)$ is a {\em pairwise balanced design} PBD$(v, K,\lambda)$ if $X$ is a set of $v$ elements and $\cB$ is a collection of subsets of $X$, called {\em blocks}, which between them contain every pair of elements of $X$ exactly $\lambda$ times.  If $K=\{k\}$, then a $\mathrm{PBD}(v,\{k\},\lambda)$ is referred to as a {\em balanced incomplete block design}, $\bibd(v,k,\lambda)$.  A collection of blocks that between them contain each point of $X$ exactly once is called a {\em resolution class} of $X$. If $\cB$ can be partitioned into resolution classes we say that the design is {\em resolvable}, and refer to the partition as a {\em resolution}.  

It is possible that a design may admit two resolutions, $\cR$ and $\cS$. If $|R_i\cap S_j| \leq 1$, for every resolution class $R_i\in \cR$ and $S_j\in \cS$, we say that these resolutions are {\em orthogonal}.  A design admitting a pair of orthogonal resolutions is called {\em doubly resolvable}.  

In the definition of a PBD or BIBD, if we relax the constraint that every pair appears exactly once to that every pair appears at most once, then we have a {\em packing} design.  Thus, a $2$-$\ghd_k(s,v;\lambda)$ may be viewed as a doubly resolvable packing design.  

A {\em group-divisible design}, $(k,\lambda)$-GDD of type $g^u$, is a triple $(X,\cG, \cB)$, where $X$ is a set of $gu$ points, $\cG$ is a partition of $X$ into $u$ subsets, called {\em groups}, of size $g$, and $\cB$ is a collection of subsets, called {\em blocks}, of size $k$, with the properties that no pair of points in the same group appears together in a block, and each pair of points in different groups occurs in exactly one block.  A $(k,1)$-GDD of type $n^k$ is called a {\em transversal design}, and denoted $\mathrm{T}(k,n)$.
Note that a transversal design can be interpreted as a packing which misses the pairs contained in the groups.  Alternatively, adding the groups as blocks, we may form from a transversal design a $\mathrm{PBD}(kn,\{k,n\})$.  PBDs formed in this way will often be used as ingredients in the construction of GHDs in later sections.

A resolvable transversal design with block size $k$ and group size $n$ is denoted $\mathrm{RT}(k,n)$.  It is well known that a set of $k$ MOLS of side $n$ is equivalent to an $\mathrm{RT}(k+1,n)$. 

\subsection{Subdesigns and holes}

The word {\em hole} is used with different meanings in different areas of design theory.  For example, a hole in a BIBD refers to a set $H$ of points such that no pair of elements of $H$ appears in the blocks, while a hole in a Latin square means an empty subsquare, so that certain rows and columns do not contain all symbols.  In the case of GHDs, three types of ``hole'' may occur, and each will be of importance later in the paper.  We thus develop appropriate terminology to differentiate the uses of the term ``hole''.

First, a {\em pairwise hole} in a GHD is a set of points $H \subseteq X$ with the property that no pair of points of $H$ appear in a block of the GHD.  Thus, a pairwise hole corresponds with the notion of a hole in a BIBD.  
  
\begin{definition}
A {\em $\ghd_k^*(s,v)$}, $\cal{G}$,   is a $2$-$\ghd_k(s,v;1)$ with a pairwise hole of size $v-s(k-1)$.  Thus the symbol set $X$ of a $\ghd_k^*(s,v)$ can be written as 
\[
X=Y \cup (S \times \{1,2,\ldots,k-1\}),
\]
where $|S|=s$, $|Y|=v-s(k-1)$, and $Y$ is a pairwise hole of $\cal{G}$.
\end{definition}
Note that this definition extends in a natural way to higher $\lambda$; however, for our purposes it will be enough to only consider the case that $\lambda=1$.
Also note that in the case that $k=2$, a $\ghd_2^*(s,v)$ is precisely the $\mathrm{H}^*(s,v)$ described by Stinson~\cite{Stinson}.

We will refer to any $\ghd_k^*(s,v)$ as having the {\em $*$-property}.  As our primary focus is on the case that $k=3$,  we will omit the subscript $k$ in this case.  Note that trivially 
any $\ghd(s,2s+1)$ (i.e.\ Kirkman square) has the $*$-property.  It is also clear that any $\ghd(s,2s+2)$ (i.e.\ DRNKTS) has the \mbox{$*$-property,} as the values of the parameters force 
there to be an unused pair of points.
%%%%%%%%%%%%%%%%%%%%%%%%%%%%%%%%
In the case of a $\ghd(s,3s)$ formed by superimposing three $\mols(s)$, the unused pairs of points are those which occur within the three groups of size $s$ in the corresponding transversal design; 
one of these groups may be taken as $Y$, so any $\ghd(s,3s)$ formed in this manner is a $\ghd^*(s,3s)$.
%%%%%%%%%%%%%%%%%%%%%%%%%%%%%%%%
In addition, the existence of a decomposable $\soma(3,6)$~\cite{PhillipsWallis} and $\soma(3,10)$~\cite{Soicher} yield the existence of $\ghd^*(6,18)$ and $\ghd^*(10,30)$.  Thus, we have the following.
\begin{theorem} \label{GHD*}
\begin{enumerate}
\item[(i)] There exists a $\ghd^*\left(\left\lfloor \frac{v-1}{2} \right\rfloor,v \right)$ for any $v\equiv 0 \pmod 3$, whenever $v=3$ or $v \geq 18$ and $v \notin N$.
\item[(ii)] There exists a $\ghd^*(n,3n)$ if and only if $n \neq 2$ or $3$.
\end{enumerate}
\end{theorem}

The second type of hole in a GHD is an {\em array hole}.  To define it, we need the concepts of trivial GHDs and subdesigns of GHDs.  A $t$-$\ghd_k(s,0;\lambda)$ is called {\em trivial}; that is, a trivial GHD is an empty $s \times s$ array.  If a $t$-$\ghd_k(s,v;\lambda)$, $\cG$, has a subarray $\cH$ which is itself a $t$-$\ghd_k(s',v';\lambda')$ for some $s'\leq s$, $v' \leq v$ and $\lambda' \leq \lambda$ and whose point set is a pairwise hole, then we say that $\cH$ is a {\em subdesign} of $\cG$.  In particular, if the subdesign $\cH$ is trivial, then we call $\cH$ an {\em array hole} in $\cG$, and say that $\cG$ has an array hole of size $s'$.

There is a third type of hole that may occur, a {\em Latin hole}, which is a set of elements $H \subseteq X$ and sets $R$, $C$ of rows and columns, respectively, such that each row in $R$ (resp.\ column in $C$) contains no points in $H$, but all points in $X \setminus H$.  
The concepts of array holes and Latin holes may coincide when there is an array hole in a Latin square, and each row and column intersecting the array hole misses the same subset $H$ of the point set.  This is often referred to in the literature as a Latin square with a hole.  
The notation $t$~$\imols(s,a)$ is used for $t~\mols(s)$, each of which has a hole of side length $a$ on the same set of positions.   Also, the notation $t~\imols(s,a,b)$ is used for a set of $t~\mols(s)$ with two disjoint holes of 
orders $a$ and $b$, with two extra properties:  (1) the two Latin holes have no symbols in common; (2)  no common row or column within the larger square intersects both array holes. 

Latin holes and array holes also both feature in the concept of frames, which are described in Section~\ref{FrameConstr}.

\section{Construction methods} \label{ConstrSection}

\subsection{Frames} \label{FrameConstr}

In this section, we discuss a useful method, that of frames, which will allow us to construct infinite families of GHDs.  The central idea has been used to construct GHDs on both ends of the 
spectrum.  However, the terminology in the literature varies: for MOLS and GHDs with few empty cells, authors often refer to HMOLS~\cite{AbelBennettGe,BennettColbournZhu,WangDu}, while for doubly 
resolvable designs authors often speak of frames~\cite{CLLM, blue book}.  We use the latter terminology, as it is more easily adaptable to more general kinds of GHDs. First we begin with a definition.  

\begin{definition} \label{Frame Definition}
Let $s$ and $v$ be integers, $X$ be a set of $v$ points,  $\{G_1, \ldots, G_n\}$ be a partition of $X$,  and $g_i=|G_i|$.  Also let $s_1, \ldots, s_n$ be non-negative integers 
with $\sum_{i=1}^{n} s_i = s$.  A {\em generalized Howell frame}, $\ghf_k$, of type $(s_1,g_1)\ldots (s_n, g_n)$ is a square array of side $s$, $A$, that has the following properties:
\begin{enumerate}
\item
Every cell of $A$ is either empty or contains a $k$-subset of elements of $X$.
The filled cells are called {\em blocks}.
\item
No pair of points from $X$ appear together in more than one block, and no pair of points in any $G_i$ appear together in any block. 
\item
The main diagonal of $A$ consists of empty $s_i\times s_i$ subsquares, $A_i$.
\item
Each row and column of the array with empty diagonal entry $A_i$ is a resolution of $X\setminus G_i$.
\end{enumerate}
\end{definition}
We will use an exponential notation $(s_1,g_1)^{n_1}\ldots (s_t,g_t)^{n_t}$ to indicate that there are $n_i$  occurrences of $(s_i, g_i)$ in the partition.  In the case that 
$k=3$ and $g_i=3s_i$, we will refer to a GHF of type $s_1s_2\cdots s_n$, and will use exponential notation $s_1^{\alpha_1} \cdots s_r^{\alpha_r}$ to indicate that there are 
$\alpha_i$ occurrences of  $s_i$.  Note that this concept of frame is analogous to the HMOLS used in~\cite{AbelBennettGe,BennettColbournZhu,WangDu}.

\begin{theorem}  \cite{AbelBennett, AbelBennettGe}  \label{Uniform Frames}
Let $h\geq 1$ and $u\geq 5$ be integers.  Then there exists a GHF of type $h^u$ except for $(h,u)=(1,6)$, and possibly for $(h,u) \in \{(1,10), (3,28), (6,18)\}$.
\end{theorem}

The following two lemmas are similar to Lemmas 2.2 and 2.3 in \cite{BennettColbournZhu} where they were written in the language of transversal designs with holes and holey MOLS. More information 
on the construction  methods used to establish them can be found in \cite{BrouwerVanRees}. For the sake of clarity, we give a brief proof here of Lemma~\ref{n=38 lemma}; the proof of
Lemma~\ref{n=68 lemma}  is similar.  

\begin{lemma} \label{n=38 lemma}
Let $h,m$  be  positive integers such that  $4~\mols(m)$ exist.   Let $v_1, v_2, \ldots, v_{m-1}$ be non-negative integers such that for 
each $i=1,2,\ldots,m-1$, there exist $3~\imols(h+v_i, v_i)$.  Then for $v=v_1+v_2+\cdots+v_{m-1}$, there exists a GHF of type $h^mv^1$. 
\end{lemma}

\begin{proof}
Take the first three $\mols(m)$ on disjoint symbol sets and superimpose them to form a $\ghd(m,3m)$, $A$.  These MOLS possess $m$ disjoint 
transversals, $T_0, T_1, \ldots, T_{m-1}$, since the cells occupied by any symbol in the fourth square form such a transversal.    Permuting 
rows and columns if necessary, we can assume that $T_0$ is on the main diagonal.  
Form an $(hm+v)\times(hm+v)$ array $A'$ as follows.  For $i,j=1,2,\ldots,m$, the subarray in rows $h(i-1)+1, \ldots, hi$ and 
columns $h(j-1)+1,\ldots,hj$ corresponds to the $(i,j)$-cell in $A$.  

We now fill the cells of $A'$.  The $h \times h$ subsquares along the diagonal, corresponding  to $T_0$, will remain empty.  Next, consider the positions 
arising from the transversals $T_1, \ldots, T_{m-1}$.  Partition the last $v$ rows into $m-1$ sets $R_1, R_2,\ldots,R_{m-1}$, with $R_{\ell}$ containing 
$v_{\ell}$ rows.  Similarly partition the last $v$ columns into $C_1, \ldots, C_{m-1}$, where $C_{\ell}$ contains $v_{\ell}$ columns.  Suppose that 
in $A$, cell $(i,j)$, containing entry $\{x_{ij}, y_{ij}, z_{ij}\}$, is in $T_{\ell}$.  In the entries of $A'$ arising from this cell, together with the 
entries in columns $h(j-1)+1,\ldots,hj$ of $R_{\ell}$, the entries in rows $h(i-1)+1,\ldots,hi$ of $C_{\ell}$ and the $v_{\ell} \times v_{\ell}$ 
subsquare formed by the intersection of $R_{\ell}$ and $C_{\ell}$, place the superimposed three $\imols(h+v_{\ell}, v_{\ell})$, 
so that the missing hole is on the $v_{\ell} \times v_{\ell}$ subsquare formed by the intersection of $R_{\ell}$ and $C_{\ell}$. 
For these superimposed MOLS,  the  symbol set is  $(\{x_{ij}, y_{ij}, z_{ij}\} \times \mathbb{Z}_h) \cup \{ \infty_{\ell, 1},\infty_{\ell, 2}, \ldots, 
\infty_{\ell, v_{\ell}}\}$, and the missing elements due to the hole are $\infty_{\ell, 1},\infty_{\ell, 2}, \ldots, \infty_{\ell, v_{\ell}}$.  

It is straightforward to verify that the resulting array $A'$ is a GHF of type $h^mv^1$.
\end{proof}

\begin{lemma}  \label{n=68 lemma}
Let $h$, $m$, $x$ and $t$ be positive integers.  Suppose there exist $3+x~\mols(m)$ and $3~\mols(h)$.  For each  $1 \leq i \leq x-1$,  let $w_i$ be a 
non-negative integer such that there exist $3~\imols(h+w_i, w_i)$.  Then if  $w=w_1+ w_2+ \cdots+w_{x-1}$,  there exists a GHF of type $h^{m-1}(m+w)^1$.
\end{lemma}

To apply Lemmas~\ref{n=38 lemma} and \ref{n=68 lemma}, we need some information on existence of $3+x~\mols(m)$, $4~\mols(m)$ and $3~\imols(h+a,a)$. 
  These existence results are given in the next two lemmas. 

\begin{lemma} \label{n=14 lemma}
There exist $4~\mols(m)$ for all integers $m \geq 5$ except for $m=6$ and possibly for $m \in \{10,22\}$.  Also, if $m$ is a prime power, there exist $m-1~\mols(m)$.  
\end{lemma}

 \begin{proof}
 See \cite{Todorov} for $n=14$, \cite{Abel} for $n=18$ and  \cite[Section III.3.6]{Handbook} for other values.
 \end{proof}

 \begin{lemma} \label{n=10 lemma} \cite[Section III.4.3]{Handbook}
  Suppose $y,a$ are positive integers with $y \geq 4a$. Then  $3~\imols(y,a)$  exist, except for $(y,a) = (6,1)$, and possibly for $(y,a) = (10,1)$.
 \end{lemma}

Most of the frames that we require in Section~\ref{TwoEmptySection} are obtained using Lemma~\ref{n=38 lemma}, but Theorem~\ref{Uniform Frames} and Lemma~\ref{n=68 lemma}
will also be used.  The majority of frames required come from the following special cases of Lemma~\ref{n=38 lemma}.

\begin{lemma} \label{Frame 7^uv}
Let $h,m,v$  be  integers with   $m \geq 5,$ $m \notin \{6,10,22\}$.  Then there exists a GHF of type $h^mv^1$ if either:
\begin{enumerate}
\item  $h=6$ and $m-1 \leq v \leq 2(m-1)$;
\item  $h\in \{7,8\}$ and $0 \leq v \leq 2(m-1)$;
\item  $h\in \{9,11\}$,  $0 \leq v \leq 3(m-1)$ and $(h,v) \neq (9,1)$.
\end{enumerate}
\end{lemma}

\begin{proof}
Apply Lemma~\ref{n=38 lemma}, and let $h$ and $m$ be as given in that lemma.   The admissible values of $m$ guarantee the existence of $4~\mols(m)$ (see Lemma~\ref{n=14 lemma}).
 The range for $v= \sum_{i=1}^{m-1} v_i$ is easily determined from the feasible values of $v_i$.  We require that $3~\imols(h+v_i, v_i)$ exist, so in light of Lemma~\ref{n=10 lemma}, we may take 
  $v_i \in \{1,2\}$ when $h=6$,  $v_i \in \{0, 1,2\}$ when $h \in \{7,8\}$,      $v_i \in \{0,2,3\}$  when $h=9$ and   $v_i \in \{0,1,2,3\}$ when $h=11$.
\end{proof}

  We next discuss how to construct GHDs from frames.  The following 
result is a generalization of the Basic Frame Construction for doubly resolvable designs; see~\cite{ACCLWW, CLLM, Lamken95}.  

\begin{theorem} [Basic Frame Construction] \label{Frame Construction}
Suppose there exists a $\ghf_k$ of type $\Pi_{i=1}^m (s_i,g_i)$.  Moreover, suppose that for $1 \leq i \leq m$, there exists a $\ghd_k(s_i+e,g_i+u)$ 
and this design contains a $\ghd_k(e,u)$ as a subdesign for $1 \leq i \leq m-1$.   Then there exists a \mbox{$\ghd_k(s+e,v+u)$,} where $s=\sum_{i=1}^m s_i$ and $v=\sum_{i=1}^m g_i$.  
\end{theorem}

\begin{proof}
Suppose the $\ghf_k$ of type $\Pi (s_i,g_i)$ has point set $X$ and groups $G_i$, $i=1,2,\ldots,m$.  Let $U$ be a set of size $u$, disjoint from $X$, and take our new point set to be $X \cup U$.  Now, 
add $e$ new rows and columns.  For each group $G_i$ with $1 \leq i \leq m-1$, fill the $s_i \times s_i$ subsquare corresponding to $G_i$, together with the $e$ new rows and columns, with a copy of the 
$\ghd_k(s_i+e,g_i+u)$, with the sub-$\ghd_k(e,u)$ (containing the points in $U$) over the $e \times e$ subsquare which forms the intersection of the new rows and columns.  See Figure~\ref{BasicFrameFigure}.  
We then delete the blocks of the sub-$\ghd_k(e,u)$.  Now, fill the $s_m \times s_m$ subsquare corresponding to $G_m$, together with the $e$ new rows and columns with the $\ghd_k(s_m+e,g_m+u)$.
\begin{figure}[ht]
\centering
\[
\begin{array}{|c|c|c|c|}  \hline
{\cellcolor[gray]{.8}}U & \ldots & {\cellcolor[gray]{.8}}- & \ldots \\ \hline
\vdots &  \ddots & \vdots &  \ddots \\  \hline
{\cellcolor[gray]{.8}}- & \ldots & {\cellcolor[gray]{.8}}G_i & \ldots  \\  \hline
\vdots &  \ddots & \vdots & \ddots \\  \hline
\end{array}
\]
\caption{\label{BasicFrameFigure}
The Basic Frame Construction.  Corresponding to the group $G_i$, the shaded cells are filled with a $\ghd_k(s_i+e,g_i+u)$.}
\end{figure}

We show that each point occurs in some cell of each row.  In the $e$ new rows, the points in $U$ appear in the blocks from the $\ghd_k(s_m+e,g_m+u)$, and the points in each $G_i$ occur in the columns corresponding to the $\ghd_k(s_i+e,g_i+u)$ for $1 \leq i \leq m$.  In a row which includes part of the $s_i \times s_i$ subsquare corresponding to the group $G_i$, the elements of $G_j$ ($j \neq i$) appear in the frame, while the elements of $G_i \cup U$ appear in the added $\ghd_k(s_i+e,g_i+u)$.  In a similar way, each element occurs exactly once in every column.
\end{proof}

We will generally use the Basic Frame Construction with $u=0$, so that our ingredient GHDs have an $e \times e$ empty subsquare.  For this special case, we have the following.
\begin{corollary} \label{FrameCorollary}
Suppose that: (1) there exists a $\ghf_k$ of type $\Pi_{i=1}^m (s_i,g_i)$; (2) for each $i\in \{1,\ldots,m-1\}$ there exists a $\ghd_k(s_i+e,g_i)$ containing a trivial subdesign $\ghd_k(e,0)$; 
and (3) there exists a $\ghd_k(s_m+e,g_m)$.  Then there exists a $\ghd_k(s+e,v)$, where $s=\sum_{i=1}^m s_i$ and $v=\sum_{i=1}^m g_i$.  
\end{corollary}

\subsection{Starter-adder method} \label{SA section}

In practice, to apply the Basic Frame Construction, Theorem~\ref{Frame Construction}, described in Section~\ref{FrameConstr}, we need to first obtain small GHDs with 
sub-GHDs to use as ingredients in the construction.  One important technique for constructing GHDs of small side length is the {\em starter-adder method}.  
See~\cite{blue book} for further background and examples.

Let $G$ be a group, written additively, and consider the set $G\times\{0,1, \ldots,c\}$, which we think of as $c$ copies of $G$ labelled by subscripts.  For 
elements $g,h\in G$, a {\em pure $(i,i)$-difference} is an element of the form $\pm(g_i-h_i)$   (i.e.\ the subscripts are the same), while a {\em mixed $(i,j)$-difference} 
is of the form $g_i-h_j$ with $i < j$.  In both cases, subtraction is done in $G$, so that $g_i-h_j = g'_i - h'_j$ if and only if $g-h=g'-h'$ in $G$, for any choice of subscripts.

\pagebreak

\begin{definition} \label{Starter-Adder Definition} 
Let $(G,+)$ be a group of order $n+x$, and  let $X=(G \times\{0,1\}) \cup (\{\infty\} \times \{1,2, \ldots, n-2x\})$.
A {\em transitive starter-adder} for a $\ghd(n+x,3n)$ is a collection of triples $\mathcal{ST}=\{S_1, S_2, \ldots, S_{n}\}$, called a {\em transitive starter},  
and a set $A=\{a_1, a_2, \ldots, a_{n}\}$, called an {\em adder},  of elements of $G$ with the following properties:
\begin{enumerate}
\item $\mathcal{ST}$ is a partition of $X$.
\item For  any $i \in \{0,1\}$, each element of $G$ occurs at most once as a pure $(i,i)$ difference  within a triple of $\mathcal{ST}$.   
Likewise, each element of $G$ occurs at most once as a mixed $(0,1)$ difference within a triple of $\mathcal{ST}$. 
If $|G|$ is even, no $g$ with order 2 in $G$ occurs as a pure $(i,i)$ difference for any $i$ in the triples of $\mathcal{ST}$.
\item Each $\infty_i$ occurs in  one triple of the form $\{\infty_i, g_0, h_1\}$, where $g,h \in G$.
\item The sets $S_1+a_1, S_2+a_2, \ldots, S_{n}+a_{n}$ form a partition of $X$.  Here  $S_j+a_j$ is the triple formed by adding $a_j$ to the non-subscript part of each non-infinite element of $S_j$,  and  $\infty_i +a_j =  \infty_i$  for any $i \in \{1,2, \ldots, n-2x\}$, $j \in \{1,2, \ldots, n\}$.
\end{enumerate}
\end{definition}

Note that if $G$ is not abelian, the element $a_j$ would be added on the right in each case.  However, in this paper we will always take the group $G$ to be $\mathbb{Z}_{n+x}$.  

Transitive starters and adders can be used to construct $\ghd(n+x,3n)$s in the following manner: label the rows and columns of the $(n+x) \times (n+x)$ array by the elements of $G$, 
then in row $i$, place the triple $S_j+i$ in the column labelled as $i-a_j$.  Thus, the first row of the array contains the blocks of the starter $\mathcal{ST}$, with 
block $S_j$ in column $-a_j$; by the definition of a starter, these blocks form a resolution class.  The remaining rows consist of translates of the first, with their 
positions shifted, so that the first column contains the blocks $S_j+a_j$ for $1 \leq j \leq n$.  By the definition of an adder, these blocks are pairwise disjoint and 
thus also form a resolution class.  The remaining columns are translates of the first.  By construction, the two resolutions (into rows and into columns) are orthogonal.

Note also that the infinite points used in a transitive starter-adder construction for a $\ghd(n+x,3n)$ form a pairwise hole of size $n-2x = 3n - 2(n+x)$.
Therefore any $\ghd(n+x,3n)$ obtained from a transitive starter-adder constrution possesses the $*$-property. We thus have the following theorem:

\begin{theorem}
If there exists a transitive starter $\mathcal{ST}$ and a corresponding adder $A$ for a $\ghd(n+x,3n)$  then there exists a $\ghd^*(n+x,3n)$.
\end{theorem}

\begin{example} \label{SAexample8}
For $s=10$, $v=24$, the following is a transitive starter and adder using $(\mathbb{Z}_{10}\times\{0,1\}) \cup \{\infty_0,\infty_1,\infty_2, \infty_3 \}$.  (The terms in square brackets give the adder.)
\[
\begin{array}{@{}*{4}l@{}}
0_{1} 6_{1} 7_{1}  [2] & 4_{1} 2_{1} 7_{0}  [3] & 6_{0} 8_{0} 8_{1}  [8] & 0_{0} 1_{0} 4_{0}  [7] \\
 \infty_0 2_{0} 3_{1}  [1] & \infty_1 5_{0} 9_{1}  [4] & \infty_2 3_{0} 1_{1}  [9] & \infty_3 9_{0} 5_{1}  [6] 
\end{array}
\]
Using this starter and adder, we obtain the following $\ghd^*(10,24)$.
\renewcommand{\arraystretch}{1.25}
\[
\arraycolsep 3.2pt
\begin{array}{|c|c|c|c|c|c|c|c|c|c|} \hline
& \infty_2 3_0 1_1 & 6_0 8_0 8_1 & 0_0 1_0 4_0 & \infty_3 9_0 5_1 & & \infty_1 5_0 9_1 & 4_1 2_1 7_0 & 0_1 6_1 7_1 & \infty_0 2_0 3_1 \\ \hline
\infty_0 3_0 4_1 & & \infty_2 4_0 2_1 & 7_0 9_0 9_1 & 1_0 2_0 5_0 & \infty_3 0_0 6_1  & & \infty_1 6_0 0_1 & 5_1 3_1 8_0 & 1_1 7_1 8_1 \\ \hline
2_1 8_1 9_1 & \infty_0 4_0 5_1 & & \infty_2 5_0 3_1 & 8_0 0_0 0_1 & 2_0 3_0 6_0  & \infty_3 1_0 7_1 & & \infty_1 7_0 1_1 & 6_1 3_1 9_0 \\ \hline
7_1 4_1 0_0 & 3_1 9_1 0_1 & \infty_0 5_0 6_1 & & \infty_2 6_0 4_1 & 9_0 1_0 1_1 & 3_0 4_0 7_0 & \infty_3 2_0 8_1 & &\infty_1 8_0 2_1 \\ \hline
\infty_1 9_0 3_1  & 8_1 5_1 1_0 & 4_1 0_1 1_1 & \infty_0 6_0 7_1 & & \infty_2 7_0 5_1 & 0_0 2_0 2_1 & 4_0 5_0 8_0 & \infty_3 3_0 9_1 & \\ \hline
 & \infty_1 0_0 4_1 & 9_1 6_1 2_0 & 5_1 1_1 2_1 & \infty_0 7_0 8_1 & & \infty_2 8_0 6_1 & 1_0 3_0 3_1 & 5_0 6_0 9_0 & \infty_3 4_0 0_1 \\ \hline
\infty_3 5_0 1_1 & & \infty_1 1_0 5_1 & 0_1 7_1 3_0 & 6_1 2_1 3_1 & \infty_0 8_0 9_1  & & \infty_2 9_0 7_1 & 2_0 4_0 4_1 & 6_0 7_0 0_0 \\ \hline
7_0 8_0 1_0 & \infty_3 6_0 2_1 & & \infty_1 2_1 6_0 & 1_1 8_1 4_0 & 7_1 3_1 4_1 & \infty_0 9_0 0_1 & & \infty_2 0_0 8_1 & 3_0 5_0 5_1 \\ \hline 
4_0 6_0 6_1 & 8_0 9_0 2_0 & \infty_3 7_0 3_1 & & \infty_1 3_1 7_0 & 2_1 9_1 5_0 & 8_1 4_1 5_1 & \infty_0 0_0 1_1 & & \infty_2 1_0 9_1 \\ \hline
\infty_2 2_0 0_1 & 5_0 7_0 7_1 & 9_0 0_0 3_0 & \infty_3 8_0 4_1 & & \infty_1 4_1 8_0 & 3_1 0_1 6_0 & 9_1 5_1 6_1 & \infty_0 1_0 2_1 & \\ \hline
\end{array} 
\]
\renewcommand{\arraystretch}{1.0}
\end{example}

Another type of starter-adder is an {\em intransitive starter-adder}. This term is defined below.

\begin{definition} \label{IntStarter-Adder Definition} 
Let $(G,+)$ be a group of order $n$, and  let $X=(G \times\{0,1,2\})$.
An {\em intransitive starter-adder} for a $\ghd(n+x,3n)$ consists of  a collection  $\mathcal{ST}$  of $n+x$ triples, called an {\em intransitive starter},
and a set $A=\{a_1, a_2, \ldots, a_{n-x}\}$  of elements of $G$,   called an {\em adder},  with the following properties:
\begin{enumerate}
\item $\mathcal{ST}$ can be partitioned into three sets $S$, $R$ and $C$ of sizes $n-x$, $x$ and $x$ respectively. We write  $S=\{S_1, S_2, \ldots, S_{n-x}\}$,
   $R=\{R_1, R_2, \ldots, R_{x}\}$ and  $C=\{C_1, C_2, \ldots, C_{x}\}$.
\item $S \cup R$ is a partition of $X$.
\item For each $i \in \{0,1,2\}$,  each element of $G$  occurs at most once as a pure $(i,i)$ difference
within a triple of $\mathcal{ST}$.   Also for each pair $(i,j) \in \{ (0,1), (0,2), (1,2)\}$, each element of $G$  occurs at most once as a mixed $(i,j)$  difference 
within a triple of $\mathcal{ST}$.  If $|G|$ is even, no pure $(i,i)$ difference $g$ with order 2 in $G$ occurs within these triples.
 \item If $S + A$ is the set of blocks   $\{S_1 + a_1, S_2+ a_2, \ldots, S_{n-x} + a_{n-x}\}$   then  $(S + A) \cup C$ is a partition of $X$.
  As before, $S_j + a_j$ denotes the block obtained by adding $a_j$ to the non-subscript part  of all elements of $S_j$.
\end{enumerate}
\end{definition}

To obtain a   $\ghd(n+x, 3n)$  from such an intransitive starter-adder, we proceed as follows. Let $\mathcal{G}$  be the required GHD, and let its the top 
left $n \times n$ subarray be $\mathcal{H}$.    Label the  first $n$ rows and columns of $\cal{G}$  by the elements of $G$.    
For $i \in G$, we then place the block $S_j + i$ in the $(i, i-a_j)$ cell of $\mathcal{H}$.

In the top row of $\mathcal{G}$ and the final  $x$ columns  (which we label as $n+1, n+2, \ldots, n+x$),  place the blocks $R_1, R_2, \ldots, R_{x}$ from $R$; then 
for $i \in G$, and $j=1,2, \ldots, x$, place $R_j + i$ in the $(i, n+j)$ cell of $\mathcal{G}$. Similarly, label the last $x$ rows of $\mathcal{G}$ as  $n+1, n+2, \ldots, n+x$.  
In the initial column and last $x$ rows of $\mathcal{G}$ we place the blocks $C_1, C_2, \ldots, C_{x}$ from $C$, and for $i \in G$, $j=1,2, \ldots, x$, we 
place $C_j + i$ in the $(n+j, i)$ cell of $\mathcal{G}$.  The bottom right $x \times x$ subarray of $\mathcal{G}$ is always an empty subarray.

\begin{theorem}
If there exists an intransitive starter $\mathcal{ST}$ and a corresponding adder $A$ for a $\ghd(n+x,3n)$  then there exists a $\ghd(n+x,3n)$ with an empty $x \times x$ sub-array.
\end{theorem}

\begin{example} \label{SAexample7}
For $n=7$, $x=2$, the following is an intransitive starter-adder over  $(\mathbb{Z}_{7} \times\{0,1,2\})$ for a $\ghd^*( 9,21)$.  (In square brackets 
we either give the adder for the corrresponding starter block, or indicate whether it belongs to $R$ or $C$.)
\[
\begin{array}{@{}*{5}l@{}}
2_{2} 3_{2} 5_{2}  [0] & 0_{0} 1_{0} 3_{0}  [6] & 6_{1} 0_{1} 2_{1}  [1] & 6_{0} 3_{1} 4_{2}  [2] & 5_0 1_{1} 6_{2}  [5]  \\
 4_0 5_{1} 0_{2}  [R] &  5_{0} 4_{1} 0_2 [C] & 2_0 4_1 1_2 [R] & 4_0 2_{1} 1_{2}  [C]
\end{array}
\]
 
Using this starter and adder, we obtain the following $\ghd^*( 9,21)$.
\renewcommand{\arraystretch}{1.25}
\[
\arraycolsep 4.5pt
\begin{array}{|c|c|c|c|c|c|c||c|c|} \hline
  2_2 3_2 5_2 & 0_0 1_0 3_0 & 5_0 1_1 6_2 &              &              &  6_0 3_1 4_2 & 6_1 0_1 2_1 & 4_0 5_1 0_2 &  2_0 4_1 1_2 \\ \hline
  0_1 1_1 3_1 & 3_2 4_2 6_2 & 1_0 2_0 4_0 & 6_0 2_1 0_2  &              &              & 0_0 4_1 5_2 & 5_0 6_1 1_2 &  3_0 5_1 2_2 \\ \hline
  1_0 5_1 6_2 & 1_1 2_1 4_1 & 4_2 5_2 0_2 & 2_0 3_0 5_0  & 0_0 3_1 1_2  &              &             & 6_0 0_1 2_2 &  4_0 6_1 3_2 \\ \hline
              & 2_0 6_1 0_2 & 2_1 3_1 5_1 & 5_2 6_2 1_2  & 3_0 4_0 6_0  & 1_0 4_1 2_2  &             & 0_0 1_1 3_2 &  5_0 0_1 4_2 \\ \hline
              &             & 3_0 0_1 1_2 & 3_1 4_1 6_1  & 6_2 0_2 2_2  & 4_0 5_0 0_0  & 2_0 5_1 3_2 & 1_0 2_1 4_2 &  6_0 1_1 5_2 \\ \hline
  3_0 6_1 4_2 &             &             & 4_0 1_1 2_2  & 4_1 5_1 0_1  & 0_2 1_2 3_2  & 5_0 6_0 1_0 & 2_0 3_1 5_2 &  0_0 2_1 6_2 \\ \hline
  6_0 0_0 2_0 & 4_0 0_1 5_2 &             &              & 5_0 2_1 3_2  & 5_1 6_1 1_1  & 1_2 2_2 4_2 & 3_0 4_1 6_2 &  1_0 3_1 0_2 \\ \hline \hline
  5_0 4_1 0_2 & 6_0 5_1 1_2 & 0_0 6_1 2_2 & 1_0 0_1 3_2  & 2_0 1_1 4_2  & 3_0 2_1 5_2  & 4_0 3_1 6_2 &             &              \\ \hline
  4_0 2_1 1_2 & 5_0 3_1 2_2 & 6_0 4_1 3_2 & 0_0 5_1 4_2  & 1_0 6_1 5_2  & 2_0 0_1 6_2  & 3_0 1_1 0_2 &             &              \\ \hline
\end{array} 
\]
\renewcommand{\arraystretch}{1.0}
\end{example}

We point out that the underlying block design for this GHD is a $(3,1)$-GDD of type $3^7$ with groups $\{t_0, t_1, t_2\}$ for $t \in \mathbb{Z}_7$.  Since it has a pairwise hole 
of size $3$, this GHD also has the $*$-property. However, GHDs obtained by the intransitive starter-adder method usually do not possess the $*$-property.

  In all examples of $\ghd(n+x,3n)$s obtained by an intransitive starter-adder  in this paper, the group $G$ will be taken as $Z_n$. Also,  
  the underlying block design for any $\ghd(n+x,3n)$ that we construct in this way  will have (in addition to the obvious automorphism of order $n$) an automorphism of order $2$.
 This automorphism maps the point $t_0$ to $t_1$, $t_1$ to $t_0$ and $t_2$ to $t_2$. It also maps any starter block with adder $a$ to the 
 starter + adder block with adder $-a$, and each block in $R$ to a block in $C$. For instance, in the previous example, the starter block $\{6_0, 3_1, 4_2\}$ 
 is mapped to the starter + adder block $\{3_0, 6_1, 4_2\} = \{5_0, 1_1, 6_2\} + 5$. 

\subsection{Designs with the $*$-property}

The following lemma generalizes constructions of Stinson~\cite{Stinson} and Vanstone~\cite{Vanstone 80} for $\ghd^*$s.  In~\cite{Stinson}, the result is given only for block 
size $k=2$, while in~\cite{Vanstone 80}, it is done for general block size in the case that $g=1$.
 
\begin{lemma}
\label{Stinson 1}
Let $g$ be a positive integer, and for each $i \in \{1,2,\ldots,g\}$, let $u_i$ be a non-negative integer.  Suppose that there exists a $\mathrm{PBD}(v,K,1)$, $(X, \cB)$, containing $g$ (not 
necessarily disjoint) resolution classes, $P_1, P_2, \ldots, P_g$.  Moreover, suppose that for every block $B\in \cB$, there exists a $\ghd_k^*(|B|, (k-1)|B|+1+u_B)$, 
where $u_B = \sum_{\{i\mid B\in P_i\}} u_i$.  Then there exists a $\ghd_k^*(v, (k-1)v+u+1)$, where $u = \sum_{i=1}^g u_i$.
\end{lemma}
\begin{proof}
We construct the resulting $\ghd_k^*$, $\cA$, on point set $(X\times \Z_{k-1})\cup I\cup \{\infty\}$, where $I = \{ \infty_{ij} \mid 1\leq i \leq g,  1\leq j \leq u_i\}$, and index 
the rows and columns of $\cA$ by $X$. For each block $B\in \cB$ we define $I_B = \{\infty_{ij}\mid B\in P_i, 1\leq j \leq  u_i\}$ 
and construct a \mbox{$\ghd_k^*(|B|, (k-1)|B|+1+u_B)$,} $\cA_B$, indexed by $B$, with point set $(B\times\Z_{k-1})\cup I_B\cup \{\infty\}$ and pairwise hole $I_B \cup \{\infty\}$. 
In this $\ghd$, the block $\{\infty, (x,0), \ldots, (x,{k-2})\}$ should appear in the $(x,x)$-entry of  $\cA_B$  for all $x \in B$.
For each cell indexed by $(x,y)\in X^2$, if $x\neq y$ there is a block with $\{x,y\}\in B$ and we place the entry from $\cA_B$ indexed by $(x,y)$ in the cell of $\cA$ indexed by $(x,y)$.
For each $x\in X$, in the diagonal $(x,x)$-entry of $\cA$ we place the block $\{\infty, (x,0), \ldots, (x,{k-2})\}$.

We now show that the resulting $\cA$ is a $\ghd_k^*$, with pairwise hole $I\cup\{\infty\}$.
We first show that no pair appears more than once by considering a pair of points in $(X\times\Z_{k-1})\cup I\cup \{\infty\}$.
If $x,y\in I\cup\{\infty\}$, it is evident that none of the elements of $I\cup \{\infty\}$ appear together in a block of $\cA$ as the elements of this set are always in the 
pairwise holes of the $\cA_B$ from which the blocks of $\cA$ are drawn, nor do they appear together in the diagonal elements. 
We now consider $(x,a) \in X\times\Z_{k-1}$.
If $y = \infty_{ij}$ then there is a unique block $B$ which contains the point $x$ in $P_i$ and $(x,a)$ and $y$ cannot appear together more than once in $\cA_B$.  
If $y=\infty$, it appears with $(x,a)$ only on the diagonal.
Finally, if $(y,b)\in X\times\Z_{k-1}$, there is a unique block $B$ which contains $x$ and $y$ and $(x,a)$ and $(y,b)$ cannot appear together more than once in $\cA_B$.

We now show that the rows of $\cA$ are resolutions of $(X\times\Z_{k-1})\cup I\cup\{\infty\}$. Consider a row indexed by $x\in X$ and an element $\alpha\in (X\times \Z_{k-1})\cup I\cup\{\infty\}$. 
If $\alpha = (x,a)$ or $\alpha = \infty$, $\alpha$ appears as a diagonal entry. If $\alpha=(y,b)$, where $y\in X\setminus \{x\}$, find the block $B$ containing $\{x,y\}$. Now, in 
the row  indexed by $x$ in $\cA_B$ the element $\alpha$ must appear, say it appears in the column indexed by $z$, then the cell indexed by $(x,z)$ of $\cA$ will contain $\alpha$.
If $\alpha = \infty_{ij}\in I$, find the block $B$ of the resolution class $P_i$ which contains $x$.  As above, in the row indexed by $x$ in $\cA_B$ the element $\alpha$ must appear, say it 
appears in the column indexed by $z$, then the cell indexed by $(x,z)$ of $\cA$ will contain $\alpha$.
A similar argument shows that the columns are also resolutions of $X$.
\end{proof}

Note that the statement of Lemma~\ref{Stinson 1} does not require the $g$ resolution classes to be disjoint, or even distinct.  In practice, however, when applying Lemma~\ref{Stinson 1}, we 
will use resolvable pairwise balanced designs in the construction and take $P_1,P_2,\ldots,P_g$ to be the distinct classes of the resolution.  Thus, for any block $B$, $u_B$ will be $u_i$, 
where $P_i$ is the parallel class containing $B$.  In particular, we will use PBDs formed from resolvable transversal designs, so we record this case  for $k=3$ in the following lemma.

\begin{lemma}
\label{Stinson 1 TD}
Suppose there exists a $\rtd(n,g)$.  For $i \in \{1, 2, \ldots, g\}$, let $u_i$ be a non-negative integer such that there exists a $\ghd^*(n,2n+1+u_i)$, and let $u_{g+1}$ be a non-negative integer 
such that there exists a $\ghd^*(g,2g+1+u_{g+1})$.  Then there exists a $\ghd^*(ng,2ng+u+1)$, where $u=\sum_{i=1}^{g+1}u_i$.  
\end{lemma}

\begin{proof}
Beginning with the $\rtd(n,g)$, construct a resolvable $\mathrm{PBD}(ng,\{n,g\},1)$ whose blocks are those of the resolvable transversal design together with a single parallel class whose 
blocks are its groups.  Note that the resolution of this design consists of $g$ parallel classes, say $P_1, P_2, \ldots, P_g$, consisting of blocks of size $n$, and one parallel 
class, $P_{g+1}$, consisting of blocks of size $g$.  The result is now a clear application of Lemma~\ref{Stinson 1}.  
\end{proof}

\section{GHDs with one or two empty cells in each row and column} \label{TwoEmptySection}

\subsection{Existence of $\ghd(n+1,3n)$} \label{t=1 section}

In~\cite{WangDu}, Wang and Du  asserted the existence of $\ghd(n+1,3n)$ for all $n \geq 7$, with at most five possible exceptions.  
However, there are  issues with some of the constructions in their paper, in particular the modified starter-adder constructions for GHD$(n+1,3n)$ with $n$ even and $10 \leq n \leq 36$.
 These constructions contained at least one starter block such that one infinite point was added to half its translates (i.e $(n+1)/2$ of them) and another infinite point was 
added to the other $(n+1)/2$ translates. However this procedure is not possible for $n$ even, since $(n+1)/2$ is then not an integer. In addition, many of their  recursive 
constructions later rely on existence of some of these faulty designs.  We thus prove an existence result for $\ghd(n+1,3n)$ here, namely Theorem~\ref{Existence_t=1}, which has no exceptions for $n \geq 7$.  Note that not all of the constructions in~\cite{WangDu} are problematic, and we will quote some results from that paper as part of our proof. 

\begin{theorem} \label{Existence_t=1}
Let $n$ be a positive integer.  There exists a $\ghd(n+1,3n)$ if and only if $n \geq 6$, except possibly for $n =6$.
\end{theorem}

If $n \leq 5$, there does not exist a $\ghd(n+1,3n)$.  In particular, if $n<3$, the obvious necessary conditions are not satisfied.  For $n=3$ or $4$, a $\ghd(n+1,3n)$ would be equivalent to 
a Kirkman square of order 9 or a doubly resolvable nearly Kirkman triple system of order 12, both of which are known not to exist; see~\cite{ACCLWW,CLLM}.  The nonexistence of 
a $\ghd(6,15)$ was stated in~\cite{WangDu} as a result of a  computer search.

We now turn our attention to existence.  The following GHDs are constructed in~\cite{WangDu}.

\begin{lemma}[Wang and Du~\cite{WangDu}] \label{revised WangDu starter adder}
There exists a $\ghd^*(n+1,3n)$ if either (1) $n=8$, or (2) $n$ is odd, and either $7 \leq n \leq 33$ or  $n=39$.
\end{lemma}

We find a number of other small GHDs by starter-adder methods.
\begin{lemma} \label{new starter adder t=1}
\begin{enumerate}
\item
There exists a $\ghd^*(n+1,3n)$ for $n \in \{14, 20,  26, 32, 38, 41, 44\}$.
\item
There exists a $\ghd(n+1,3n)$ for $n \in \{10, 12,  16, 18, 22, 28, 30, 36, 46\}$.
\end{enumerate}
\end{lemma}

\begin{proof}
Starters and adders for these GHDs can be found in Appendix~\ref{1EmptyAppendix}.
We use a transitive starter and adder for the values in (1), and an intransitive one for the values in (2).
\end{proof}

\begin{lemma} \label{StarterAdderMod2}
There exists a $\ghd^*(n+1,3n)$ for $n=37$.
\end{lemma}

\begin{proof}
We give a transitive starter and adder, modifying how we develop some of the blocks; this method is also used in~\cite{WangDu}.  We develop the subscripts of $\infty_0$ and $\infty_1$ modulo 2.  
The points $\infty_2$ and $\infty_3$ are treated similarly: $\{\infty_2,a_i,b_i\}$ gives $\{\infty_3, (a+1)_i, (b+1)_i\}$ in the next row, while $\{\infty_3,c_j,d_j\}$ yields 
$\{\infty_2,(c+1)_j,(d+1)_j\}$ in the next row.  Note that this ``swapping'' of infinite points in subsequent rows allows us to include a pure difference in blocks containing an infinite point.  
That the block of the starter containing $\infty_0$ (resp.\ $\infty_2$) also has points of the form $a_0$, $b_0$ where $a$ and $b$ have different parities and the block 
containing $\infty_1$ (resp.\ $\infty_3$) also has points of the form $c_1$, $d_1$, where $c$ and $d$ have different parities, combined with the fact that $n+1$ is even, ensures that no pair 
of points is repeated as we develop.  The points $\infty_4, \infty_5, \ldots, \infty_{34}$ remain fixed as their blocks are developed.  

The starter blocks are given below, with the corresponding adders in square brackets.
\[
\arraycolsep 1.9pt
\begin{array}{@{}*{6}l@{}}
0_0  7_0 13_0         [0]  & 10_1 14_1 20_1       [12] & \infty_0 37_0  2_0    [18] & \infty_1  32_1  1_1    [10]  & \infty_2 33_0 34_0    [32]  & \infty_3 2_1   7_1    [22] \\
\infty_4 25_0 25_1    [6]  & \infty_5 8_0  9_1    [14] & \infty_6 16_0 18_1      [26] & \infty_7 26_0 29_1     [4] & \infty_8 18_0 22_1     [17] & \infty_9 21_0 26_1    [28] \\
\infty_{10} 24_0 30_1 [15] & \infty_{11} 3_0 11_1    [23] & \infty_{12} 4_0 13_1  [2] & \infty_{13} 5_0 15_1  [31] & \infty_{14} 10_0 21_1  [27] & \infty_{15} 12_0 24_1 [35] \\ 
\infty_{16} 28_0  3_1  [34]  & \infty_{17} 19_0 34_1 [21] & \infty_{18} 30_0  8_1 [33] & \infty_{19} 27_0  6_1 [19] & \infty_{20} 15_0 33_1  [3] & \infty_{21} 23_0  4_1  [9] \\ 
 \infty_{22} 11_0 31_1 [37] & \infty_{23} 22_0  5_1  [30] & \infty_{24} 35_0 19_1  [8] & \infty_{25} 31_0 16_1 [36] & \infty_{26} 14_0  0_1 [20] & \infty_{27} 1_0 27_1  [11] \\ 
 \infty_{28} 9_0 36_1  [7] &  \infty_{29} 6_0 37_1  [13] & \infty_{30} 29_0 23_1  [24] & \infty_{31} 17_0 12_1 [16] & \infty_{32} 32_0 28_1 [29] & \infty_{33} 20_0 17_1  [1] \\ 
 \infty_{34} 36_0 35_1   [5]
\end{array}
\]
\end{proof}

\begin{lemma}\label{StarterAdderMod5}
There exists a $\ghd(n+1,3n)$ for $n\in \{24,34\}$.
\end{lemma}

\begin{proof}
We use a similar modification of the transitive starter-adder method to that in Lemma~\ref{StarterAdderMod2}.  In this case, the subscripts 
of $\infty_0, \infty_1, \infty_2, \infty_3, \infty_4$ are developed modulo $5$ as we develop the remaining points modulo ${n+1}$.

For $n=24$, the starter and adder are as follows: 
\[ 
\arraycolsep 2.2pt
\begin{array}{@{}*{6}l@{}}
1_0  4_0 10_0 [19] & 0_1  7_1 15_1  [24] & \infty_0 20_0 21_1 [5] & \infty_1 12_0 19_0 [20] & \infty_2 5_1 19_1  [10] & \infty_3 2_1  3_1 [15] \\
\infty_4 13_0 21_0  [0] & \infty_5 7_0 12_1 [8] & \infty_6 6_0 13_1  [6] & \infty_7 23_0  6_1  [7] & \infty_8 2_0 11_1   [17] & \infty_9 24_0  9_1  [2] \\
\infty_{10} 18_0  4_1 [4] & \infty_{11} 14_0  1_1  [21] & \infty_{12} 5_0 18_1  [23] & \infty_{13} 8_0 22_1 [3] & \infty_{14} 9_0 24_1  [22] & \infty_{15} 0_0 17_1 [18] \\
\infty_{16} 15_0  8_1 [1] & \infty_{17} 3_0 23_1  [14] & \infty_{18} 17_0 14_1  [16] & \infty_{19} 22_0 20_1  [12] & \infty_{20} 11_0 10_1   [13] & \infty_{21} 16_0 16_1 [11]
\end{array}
\]

For $n=34$, the starter and adder are as follows:
\[
\arraycolsep 1.9pt
\begin{array}{@{}*{6}l@{}}
0_0 11_0 17_0  [4] & 10_1 21_1 22_1 [24] & \infty_0 6_0 13_1 [0] & \infty_1 20_1 27_1 [10] & \infty_2 15_0 34_0 [15] & \infty_3 18_0 27_0 [30] \\
\infty_4 11_1 14_1 [5] & \infty_5 29_0 29_1 [26] & \infty_6 5_0  6_1 [11] & \infty_7 9_0 12_1 [9] & \infty_8 1_0  5_1 [2] & \infty_9 2_0  7_1  [34] \\
\infty_{10} 20_0 28_1 [8] & \infty_{11} 7_0 16_1 [16] & \infty_{12} 8_0 18_1 [17] & \infty_{13} 4_0 15_1 [29] & \infty_{14} 12_0 24_1 [23] & \infty_{15} 21_0 34_1 [6] \\
\infty_{16} 19_0 33_1 [33] & \infty_{17} 23_0  3_1 [20] & \infty_{18} 22_0  4_1 [25] & \infty_{19} 26_0  9_1 [19] & \infty_{20} 24_0  8_1 [18] & \infty_{21} 16_0  1_1 [3] \\
\infty_{22} 10_0 31_1 [22] & \infty_{23} 13_0  0_1 [27] & \infty_{24} 14_0  2_1 [12] & \infty_{25} 33_0 23_1 [1] & \infty_{26} 28_0 19_1 [31] & \infty_{27} 25_0 17_1 [21] \\
\infty_{28} 32_0 25_1 [32] & \infty_{29} 3_0 32_1 [28] & \infty_{30} 30_0 26_1 [7] & \infty_{31} 31_0 30_1 [13]
\end{array}
\]

\end{proof}

Most of our GHDs are constructed using the Basic Frame Construction with $u=0$ (Corollary~\ref{FrameCorollary}).  Note that in 
every  $\ghd(n+1,3n)$ there is exactly one empty cell in each row and column;  thus every $\ghd(n+1,3n)$ contains a $\ghd(1,0)$ as 
a subdesign.  We therefore do not need to separately verify existence of this subdesign  in the construction.

\begin{lemma} \label{uniform t=1}
There is a $\ghd(n+1,3n)$ for $n \in \{35,40,45,49,50,51,54,55\}$.
\end{lemma}

\begin{proof}
For $n=35,40,45,49,54$ and $55$, there exist (by Theorem~\ref{Uniform Frames})  GHFs of types $7^5$, $8^5$,  $9^5$  $7^7$, $9^6$ and $11^5$.  
For $n=50,51$, there exist (by Lemma~\ref{n=68 lemma}, with $h=m=7$, $x=1$ and either $w_1 = w = 1$ or $w_1 = w = 2$)
 GHFs of types $7^6 8^1$ and $7^6 9^1$.   Since there exist $\ghd(8,21)$, $\ghd(9,24)$, $\ghd(10,27)$ and $\ghd(12,33)$ 
 by Lemma~\ref{revised WangDu starter adder}, the result follows by Corollary~\ref{FrameCorollary}.
\end{proof}

\begin{lemma} \label{t=1 recursion}
If $ n \in \{42,43,47,48,52,53\}$ or $n \geq 56$, then there is a $\ghd(n+1,3n)$.
\end{lemma}

\begin{proof}

First, suppose that $n \geq 84$.  Here, we can write $n=7m+v$, where $m \geq 11$ is odd, 
$v$ is odd and $7 \leq v \leq 20 \leq 2m-2$. By Lemma~\ref{Frame 7^uv}, there is a GHF of type $7^m v^1$.  
Since there exists a $\ghd(8,21)$ and a $\ghd(v+1,3v)$ (by Lemma~\ref{revised WangDu starter adder} 
  or Lemma~\ref{new starter adder t=1}), a $\ghd(n+1,3n)$ exists by Corollary~\ref{FrameCorollary}.

For the remaining values of $n$,  we give in Table~\ref{ghd(n+1,3n) Frame Table}  a frame of appropriate type  $h^mv^1$,  
where $n= hm+v$; these frames all exist by Lemma~\ref{Frame 7^uv}.   Together with the existence of a $\ghd(h+1,3h)$ and 
a $\ghd(v+1,3v)$ (by Lemma~\ref{revised WangDu starter adder} or Lemma~\ref{new starter adder t=1}), 
these frames  give the required GHD$(n+1,3n)$ by Corollary~\ref{FrameCorollary}.

\begin{table}[htbp]
\caption{Frames used for $\ghd(n+1,3n)$ with $n \in \{42,43,47,48,52,53\}$ and $56 \leq n \leq 83$.}
\centering
 \begin{tabular}{ccccccccccc} \hline
 $n=hm+v$ & $h$ & $m$ & $v$   & \hspace{0.5cm} &   $n=hm+v$ & $h$ & $m$ & $v$  \\ \hline
 42--43    & 7  & 5 & 7--8   &                  &    47--48   &  8 & 5 & 7--8   \\
 52--53    & 9  & 5 & 7--8   &                  &    56--61   &  7 & 7 & 7--12  \\
  62      & 11 & 5 & 7     &                  &    63--69   &  7 & 8 & 7--13  \\
 70--79    & 7  & 9 & 7--16  &                  &    80--83   &  8 & 9 & 8--11  \\ \hline
 \end{tabular}

\label{ghd(n+1,3n) Frame Table}
\end{table}
\end{proof}

Taken together, Lemmas~\ref{revised WangDu starter adder}--\ref{StarterAdderMod5} and \ref{uniform t=1}--\ref{t=1 recursion} prove Theorem~\ref{Existence_t=1}.

\subsection{Existence of $\ghd(n+2,3n)$} \label{t=2 section}

In this section, we consider the existence of $\ghd(n+2,3n)$, which have two empty cells in each row and column, and prove the following theorem.
\begin{theorem} \label{Existence_t=2}
Let $n$ be a positive integer.  Then there exists a $\ghd(n+2,3n)$ if and only if $n \geq 6$.
\end{theorem}

Note that if $1 \leq n \leq 4$, the necessary conditions for the existence of a $\ghd(n+2,3n)$ are not satisfied.  Moreover, for $n=5$, there is 
no $\ghd(7,15)$~\cite{Mathon Vanstone}.  Thus it suffices to consider the case that $n \geq 6$.  

For relatively small values of $n$, we construct $\ghd(n+2,3n)$ mainly by starter-adder methods.  These will then be used in the Basic Frame Construction 
with $u=0$ (Corollary~\ref{FrameCorollary}) to give the remaining GHDs.  
\begin{lemma} \label{Small Cases}
\rule{0ex}{0ex}
    For all $n \in \{6, \ldots, 29\} \cup  \{31, \ldots, 34\}  \cup \{39,44\}$,  a $\ghd^*(n+2,3n)$ exists. 
Moreover, if $n$ is even, then there exists  such a design containing a $\ghd(2,0)$ as a subdesign.
\end{lemma}

\begin{proof} 
As mentioned in Section~\ref{DefnSection}, there exists a $\ghd(8,18)$~\cite{Finland}.  As such a design is equivalent to a $\mathrm{DRNKTS}(18)$, it has the $*$-property.  Note that 
the $\ghd^*(8,18)$ exhibited in Example~\ref{FinlandExample} has a $2\times2$ empty subsquare, i.e\ a sub-$\ghd(2,0)$.  For the remaining values of $n$, a  transitive starter  and adder for 
a $\ghd^*(n+2,3n)$ can be found in Appendix~\ref{starters and adders}.  Note that the group used is $\mathbb{Z}_{n+2}$.  For $n$ even, the adder does not contain 0 or $(n+2)/2$, 
thus ensuring that the $(0,0)$-cell and $(0,(n+2)/2)$-cell are empty, and as we develop, the $((n+2)/2,0)$-cell and $((n+2)/2,(n+2)/2)$-cell will also be empty, yielding a sub-$\ghd(2,0)$. 
\end{proof}

We remark that a $\ghd^*(10,24)$ with a sub-$\ghd(2,0)$ can also be obtained from~\cite[Table 8]{DuAbelWang}.

\begin{lemma} \label{Odd2hole}
\rule{0ex}{0ex}
    For $n \in \{7,9,11\}$, there exists a $\ghd(n+2,3n)$    containing a $\ghd(2,0)$ as a subdesign. Further, when $n=7$, this GHD has the $*$-property. 
\end{lemma}

\begin{proof} 
  For these values of $n$, an  intransitive starter  and adder for a $\ghd^*(n+2,3n)$ can be found in Appendix~\ref{starters and adders}.  When $n=7$, 
this GHD is also given in Example~\ref{SAexample7}; there we indicated that this one has the $*$-property. 
\end{proof}

Our main recursive construction for $n \geq 84$ uses frames of type  $7^mv^1$ from Lemma~\ref{Frame 7^uv}.  For smaller $n$, we also use frames of types $h^mv^1$ with $h \in \{6,8,9,11\}$,
which also come from Lemma~\ref{Frame 7^uv}.  Note that in light of Lemma~\ref{Small Cases}, in order to prove Theorem~\ref{Existence_t=2}, we need to obtain a $\ghd(n+2,3n)$ for 
each  $n \in \{30, 35, 36, 37, 38, 40, 41, 42, 43\}$ and for all $n \geq 45$.

\begin{lemma} \label{GHDs from uniform frames}
There exists a $\ghd(n+2,3n)$ for $n \in \{30,$ $35,$ $36,$ $37,$ $38,$ $40,$ $41,$ $46,$ $49,$ $50,$ $51,$ $54,$ $55\}$.
\end{lemma}

\begin{proof}
We apply Corollary~\ref{FrameCorollary} to construct the required designs.  For $n=30,$ $35$, $36$, $40$, $49$, $54$ and $55$,  we use  GHFs of types 
$6^5$, $7^5$, $6^6$,  $8^5$, $7^7$, $9^6$ and $11^5$ respectively. These frames all exist by Theorem~\ref{Uniform Frames}.   For $n=37$, $38$, 
$41$ and $46$, we use GHFs of types $6^5 7^1$, $6^5 8^1$, $7^5 6^1$ and  $8^5 6^1$  respectively (these all exist by  Lemma~\ref{Frame 7^uv}). For 
$n=50$ and $51$, we use GHFs of types $7^6 8^1$ and $7^6 9^1$ respectively (these  exist by  Lemma~\ref{n=68 lemma}  with $h=m=7$, $x=2$ and 
$w_1 = w = 1$ or $2$).  Since there exist $\ghd(8,18)$,  $\ghd(9,21)$, $\ghd(10,24)$, $\ghd(11,27)$ and  $\ghd(13,33)$ each containing 
a sub-$\ghd(2,0)$ (by Lemmas~\ref{Small Cases} and  \ref{Odd2hole}) the result follows.
\end{proof}

\begin{lemma} \label{Intermediate range}
There exists a $\ghd(n+2,3n)$ if $n \in \{42, 43, 47, 48, 52, 53\}$  or $n \geq 56$.
\end{lemma}

\begin{proof}
We write $n$ in one of the forms   $7m+v$, $8m+v$, $9m+v$  or $11m+v$ (where (1) $m \geq 5$ and either $m$ is odd or $m=8$, (2) $7 \leq v \leq$ min$(2(m-1), 20$)) 
in the same manner as we did in Lemma~\ref{t=1 recursion}.  We then construct a frame of type  $7^mv^1$, $8^mv^1$, $9^m v^1$ or $11^mv^1$ from 
  Lemma~\ref{Frame 7^uv}  together with a $\ghd(9,21)$, $\ghd(10,24)$, $\ghd(11,27)$ or $\ghd(13,33)$ (each containing 
a $\ghd(2,0)$ as a subdesign) and a $\ghd(v+2,3v)$   (all of these exist by Lemma~\ref{Small Cases}  or Lemma~\ref{Odd2hole}). 
Applying  Corollary~\ref{FrameCorollary},  using these frames and GHDs then produces the required $\ghd(n+2,3n)$.  
\end{proof}

Lemmas~\ref{Small Cases}, \ref{Odd2hole},  \ref{GHDs from uniform frames} and  \ref{Intermediate range}  together 
prove Theorem~\ref{Existence_t=2}.

\section{GHDs across the spectrum} \label{Spectrum}

In a non-trivial $\ghd(s,v)$ (i.e.\ where $v\neq 0)$, we require that $2s+1 \leq v \leq 3s$.  A $\ghd(s,3s)$ has no empty cells, while a $\ghd(s,2s+1)$ has $(s-1)/3$ empty cells in each row and column.  
Noting that $\lim_{s\rightarrow\infty} \frac{s-1}{3s} = \frac{1}{3}$, we see that the proportion of cells in a given row or column which are empty falls in the interval $[0,1/3)$.  In this section, 
we prove that for any $\pi \in [0,5/18]$, there is a GHD whose proportion of empty cells in a row or column is arbitrarily close to $\pi$.  

Our main tool in this section is Lemma~\ref{Stinson 1} and its variant, Lemma~\ref{Stinson 1 TD}.  As an ingredient for this construction, we require GHDs which have the $*$-property.  
We note that GHDs constructed by the Basic Frame Construction do not always have the $*$-property, even if the input designs do.\footnote{%
GHDs constructed by the Basic Frame Construction do have the $*$-property if (1) the frame used is a $\ghf_k$ of type $(s_1,g_1), \ldots, (s_n, g_n)$ with $g_i = (k-1)s_i$ for $i=1, \ldots, n$ and (2) for $i=1, \ldots n$, the input designs are $\ghd(s_i, g_i + t)$s with a pairwise hole of size $t$ for some $t$.  However, condition (1) is not satisfied by the frames in this paper.}
Thus, in general, we cannot use the results of Section~\ref{TwoEmptySection} for this purpose.  However, as previously noted, those GHDs constructed by transitive starter-adder methods do have the $*$-property.

\begin{lemma} \label{1-empty-cell}
Let $m \geq 6$.  There exists a $\mathrm{GHD}^*(2^m,3\cdot 2^m-3)$.
\end{lemma}

\begin{proof}
Since $2^{m-3}$ is a prime power and $m \geq 6$, there is an $\rtd(8,2^{m-3})$, with disjoint parallel classes $P_1, P_2, \ldots, P_{2^{m-3}}$.  Form a $\pbd(2^m,\{2^{m-3},8\},1)$ by adding a single 
parallel class $P_{2^{m-3}+1}$ consisting of the groups of the RTD.  For the parallel class $P_{2^{m-3}+1}$, set $u_{2^{m-3}+1} =2^{m-3}-1$.  For parallel classes $P_i$ with $1  \leq i \leq  2^{m-3}-1$, 
let $u_i=7$,  and for parallel  class $P_{2^{m-3}}$, let $u_{2^{m-3}}=4$.  Since there exist a $\mathrm{GHD}^*(2^{m-3},3\cdot 2^{m-3})$ and a $\mathrm{GHD}^*(8,24)$ (both by Theorem~\ref{GHD*}(ii)), as 
well as  a $\mathrm{GHD}^*(8,21)$ (by Lemma~\ref{revised WangDu starter adder}), and 
\[
2 \cdot 2^m+1 + \left((2^{m-3}-1) + (2^{m-3}-1)(7)+4\right) = 3 \cdot 2^m-3
\]
it follows by Lemma~\ref{Stinson 1 TD} that there exists a $\ghd^*(2^m,3 \cdot 2^m-3)$.
\end{proof}

\begin{lemma} \label{2-empty-cell}
Let $m \geq 6$.  There exists a $\mathrm{GHD}^*(2^m,3\cdot 2^m-6)$.
\end{lemma}

\begin{proof}
The proof is similar to that of Lemma~\ref{1-empty-cell}, except that we take $u_{2^{m-3}}=1$ rather than 4, which requires a $\ghd^*(8,18)$ (given in Example~\ref{FinlandExample}) rather than a $\ghd^*(8,21)$.
\end{proof}

With these ingredients in hand, we now construct GHDs with side length a power of 2.

\begin{lemma} \label{Power-2-even}
Let $m \geq 7$ be odd, and let $A=\frac{5}{36}\cdot 2^{2m}+\frac{5}{18}\cdot 2^m-\frac{19}{9}$.  For all $0 \leq \alpha \leq A$,
there exists a $\mathrm{GHD}^*(2^{2m},3 \cdot 2^{2m}-6\alpha)$.  
\end{lemma}

\begin{proof}
Since $2^m$ is a prime power, there is a $\pbd(2^{2m},\{2^m\},1)$ (an affine plane of order $2^m$); note that the number of parallel classes is $2^m+1$.  Let $x$ and $y$ be integers with $0 \leq x, y \leq 2^m+1$ and $x+y \leq 2^{m}+1$.  In Lemma~\ref{Stinson 1}, for $x$ parallel classes take $u_i=1$, for $y$ parallel classes take $u_i=2^m-1$, and for the remaining $2^m+1-x-y$ parallel classes take $u_i=2^m-7$.  Note that there exist a $\ghd^*(2^m, 2 \cdot 2^m + 2)$ (by Theorem~\ref{GHD*}(i)), a $\ghd^*(2^m, 3 \cdot 2^m)$ (by Theorem~\ref{GHD*}(ii)) and a $\ghd^*(2^m, 3 \cdot 2^m-6)$ (by Lemma~\ref{2-empty-cell}).  Thus, by Lemma~\ref{Stinson 1}, there is a $\ghd^*(2^{2m}, 2 \cdot 2^{2m}+1+x+y(2^m-1)+(2^m+1-x-y)(2^m-7))$.  Note that the number of points is
\[
2 \cdot 2^{2m}+1+x+y(2^m-1)+(2^m+1-x-y)(2^m-7) = 3 \cdot 2^{2m} -6(1+2^{m})-x(2^m-8)+6y.
\]
Let $f(m,x,y)$ denote this number of points.  For a fixed $x$, varying $y$ between 0 and $2^m+1-x$ gives all values of the number of points congruent to $0\pmod{6}$ between $f(m,x,0)$ and $f(m,x,2^m+1-x)$.  Noting that for fixed $m$, $f(m,x,y)$ is linear in $x$ and $y$, and solving $f(m,x,0)=f(m,x+1,2^m+1-(x+1))$ for $x$, we obtain the solution $x_0=\frac{5}{6}2^m+\frac{4}{3}$ (which is an integer since $m$ is odd).  For $x \leq x_0$, we have that $f(m,x,0) \leq f(m,x+1,2^m+1-(x+1))$, which means that we cover all possible values for the number of points congruent to $0\pmod{6}$ from $f(m,x_0+1,0)$ to $3 \cdot 2^{2m}$.  Moreover, 
\begin{eqnarray*}
f(m,x_0+1,0) &=& 3\cdot 2^{2m}-6(1+2^m)-\left(\frac{5}{6} \cdot 2^m+\frac{7}{3}\right)(2^m-8) \\
&=& 3 \cdot 2^{2m} - \left(6 + 6 \cdot 2^m + \frac{5}{6} \cdot 2^{2m}-\frac{20}{3} \cdot 2^m  +\frac{7}{3} \cdot 2^m - \frac{56}{3}\right) \\
&=& 3 \cdot 2^{2m} - \left( \frac{5}{6} \cdot 2^{2m} + \frac{5}{3} \cdot 2^{m} - \frac{38}{3}\right) \\
&=& 3 \cdot 2^{2m} - 6 \left( \frac{5}{36} \cdot 2^{2m} + \frac{5}{18} \cdot 2^m - \frac{19}{9} \right) \\
&=& 3 \cdot 2^{2m} - 6A,
\end{eqnarray*}  
and so the result is verified.
\end{proof}

\begin{lemma} \label{Power-2-odd}
Let $m \geq 7$ be odd, and let $A'=\frac{5}{36} \cdot 2^{2m} + \frac{1}{9} \cdot 2^m - \frac{23}{18}$.  For all $1 \leq \alpha \leq A'$, there exists a $\ghd(2^{2m},3\cdot 2^{2m}-6\alpha+3)$.
\end{lemma}

\begin{proof}
The proof is similar to that of Lemma~\ref{Power-2-even}, except that on one parallel class we take $u_i=2^m-4$ instead of $u_i=2^m-7$.  This requires a $\ghd^*(2^m,3 \cdot 2^m-3)$, which exists by Lemma~\ref{1-empty-cell}.
\end{proof}

In Lemma~\ref{Power-2-even}, the number of empty cells in each row of the $\mathrm{GHD}^*(2^{2m},3 \cdot 2^{2m}-6\alpha)$ is $2\alpha$, while in Lemma~\ref{Power-2-odd}, the number of 
empty cells in each row of the \mbox{$\ghd^*(2^{2m},3 \cdot 2^{2m}-6\alpha+3)$} is $2\alpha-1$.  Note that in Lemmas~\ref{Power-2-even} and~\ref{Power-2-odd}, for $\alpha=A$ and $\alpha=A'$, 
respectively, the number of points is less than $3 \cdot 2^{2m} - \frac{5}{6} \cdot 2^{2m}$, so that we have GHDs of side length $2^{2m}$ for any number of empty cells per row 
between 0 and $\frac{5}{18} \cdot 2^{2m}$, giving proportions of empty cells per row or column between 0 and $\frac{5}{18}$. 
Approximating any real number $\pi \in [0,5/18]$ by a dyadic rational, this means that we can now construct a $\ghd$ such that the proportion of empty cells in a row is arbitrarily close to $\pi$.  

\begin{theorem} \label{proportion}
Let $\pi \in [0,5/18]$.  For any $\varepsilon>0$, there exists an odd integer $m$ and an integer $v$ for which there exists a $\mathrm{GHD}^*(2^{2m},v)$, $\mathcal{D}$, such that the proportion $\pi_0$ 
of empty cells in each row or column of $\mathcal{D}$ satisfies $|\pi-\pi_0|<\epsilon$.
\end{theorem}

\begin{proof}
Given $\pi$ and $\varepsilon$, there exists $m_0$ such that for all $m_1>m_0$, $|\pi-\lfloor 2^{m_1} \pi \rfloor / 2^{m_1}|<\varepsilon$. 
Let $m$ be an odd integer with $2m>m_0$, and let $\pi_0=\lfloor 2^{2m} \pi \rfloor / 2^{2m}$.  Note that if $\pi=0$, then $\pi_0=0$.  Otherwise, since $\pi-\frac{1}{2^{2m}} < \pi_0 \leq \pi$, we 
may also choose $m$ sufficiently large to ensure that $\pi_0 \in [0,\frac{5}{18}]$.  Thus Lemmas~\ref{Power-2-odd} and~\ref{Power-2-even} enable us to construct a $\ghd^*$ with side 
length $2^{2m}$ whose proportion of empty cells in a row is $\pi_0$. 
\end{proof}

Theorem~\ref{proportion} shows that we can find a $\ghd$ with an arbitrary proportion of empty cells across five-sixths of the interval $[0,1/3)$ of possible proportions.  This improves on previous 
work which has shown existence only very close to the ends of the spectrum.  The impact of this result is discussed further in Section~\ref{concl}.  

We remark also that GHDs exist throughout the spectrum with more general side lengths than powers of~$2$.  
First, the methods of Lemmas~\ref{1-empty-cell}--\ref{Power-2-odd} work for primes other than~$2$, provided we can find appropriate ingredient $\ghd$s.  Also, by a straightforward 
generalization of  the Moore--MacNeish product construction for MOLS~\cite{MacNeish22, Moore}, the existence of a $\ghd(s,v)$ and $3~\mols(n)$ implies that there exists 
a $\ghd(ns,nv)$. See, for instance, Theorem 2.6 in \cite{YanYin} for a similar result.  By applying this construction to the results of Lemmas~\ref{Power-2-even} and~\ref{Power-2-odd}, we 
can find GHDs with an arbitrary proportion of empty  cells (for proportions in $[0,5/18]$) for side lengths of the form $2^{2m}n$.

\section{Conclusion} \label{concl}

The concept of generalized Howell designs brings together various classes of designs, from doubly resolvable BIBDs on one side of the spectrum to MOLS and SOMAs on the other.  In this paper, 
we have defined generalized Howell designs in a way that encompasses several previously studied generalizations of Howell designs, and have attempted to unify disparate terminology for techniques 
used on both ends of the spectrum.  

In Section~\ref{ConstrSection}, we described several construction techniques for $\ghd$s, several of which generalize known constructions for Howell designs, doubly resolvable designs and MOLS.  
These construction techniques were used in Section~\ref{TwoEmptySection} to settle existence of $\ghd(s,v)$ in the case that the number of empty cells in each row and 
column is one or two,  with one  possible exception (a $\ghd(7,18)$) (Theorems~\ref{Existence_t=1} and \ref{Existence_t=2}).
The existence of $\ghd(s,v)$ with $e$ empty cells, where $e \in \{3, 4, \ldots, (s-3)/3 \} \setminus \{(s-6)/3\}$
remains open in general (although in \cite{DuAbelWang}, the case  $e=(s-4)/3$  was solved for $e$ even with  15 possible exceptions).
$e=0$ (a $\ghd(6,18)$) $e=1$ (a $\ghd(9,24)$) and $e=2$ (a $\ghd(12,30)$) are now known  to exist. (see Theorems~\ref{618}, \ref{Existence_t=1} and \ref{Existence_t=2}).)    
A simpler interim result would be to show existence for an interval of $e$-values 
of a given fixed length.  We conjecture that there exists a $\ghd(s,v)$ whenever the obvious necessary conditions are satisfied, with at most a small number of exceptions for each $e$. 

The main result of Section~\ref{Spectrum} is that for any $\pi \in [0,5/18]$, there exists a $\ghd$ whose proportion of cells in a given row or column which are empty is arbitrarily close 
to $\pi$.  This is a powerful result.  While it does not close the existence spectrum, it does provide strong evidence that this should be possible.  Previous work has focused on the two 
ends of the spectrum: Kirkman squares and DRNKTSs at one end, and MOLS, SOMAs, and GHDs with one empty cell per row/column at the other; Theorem~\ref{proportion} shows existence of 
GHDs across five-sixths of the spectrum.
The techniques of Section~\ref{Spectrum} can be used to give some examples of $\ghd$s with proportion greater than $5/18$, but necessarily bounded away from $1/3$.  It remains a challenging 
open problem to show that there exist $\ghd$s whose proportion of empty cells per row or column can be arbitrarily close to any element of $[0,1/3)$.

\section{Acknowledgments}
The authors would like to thank Esther Lamken for a number of useful comments
and in particular for suggesting  the intransitive starter-adder method,
which was used for several of the smaller GHDs in this paper.

\begin{appendices}
\section{Starters and adders for small \mbox{\boldmath $\ghd(n+1,3n)$}}
\label{1EmptyAppendix}

First we give those obtained by transitive starters and adders:

\begin{example} For $n=14$:
\[
\begin{array}{@{}*{6}l@{}}
0_0  3_0  5_0             [10] & 1_1  3_1  7_1            [ 5] & \infty_0  4_0  4_1        [12] &  \infty_1  8_0  9_1       [ 1] & \infty_2   9_0  11_1     [ 2] & \infty_3 11_0 14_1       [ 8] \\   
\infty_4 13_0  2_1        [ 9] & \infty_5 10_0  0_1       [ 4] & \infty_6 14_0  5_1        [13] & \infty_7   6_0 13_1       [11] & \infty_8   1_0  10_1     [ 7] & \infty_9  2_0 12_1       [ 3] \\
\infty_{10} 12_0  8_1     [ 6] & \infty_{11}  7_0  6_1    [14] 
\end{array}
\]

\end{example}

\begin{example} For $n=20$:
\[
\begin{array}{@{}*{6}l@{}}
0_0  1_0  3_0             [14] &  2_1  6_1  8_1           [ 7] & \infty_0  8_0  9_1        [ 2] &  \infty_1  9_0 11_1       [12] & \infty_2   7_0  10_1     [11] & \infty_3 11_0 15_1       [ 5] \\   
\infty_4 13_0 18_1        [17] & \infty_5 18_0  3_1       [ 4] & \infty_6 10_0 17_1        [10] & \infty_7  14_0  1_1       [18] & \infty_8  12_0   0_1     [16] & \infty_9 15_0  4_1       [ 8] \\
\infty_{10}  5_0 16_1     [ 1] & \infty_{11} 16_0  7_1    [ 3] & \infty_{12} 20_0 12_1     [13] & \infty_{13}  4_0 19_1     [20] & \infty_{14} 19_0 14_1    [15] & \infty_{15} 17_0 13_1    [ 9] \\
\infty_{16}  2_0 20_1     [ 6] & \infty_{17}  6_0  5_1    [19]  
\end{array}
\]

\end{example}

\begin{example} For $n=26$:
\[
\begin{array}{@{}*{6}l@{}}
0_0  6_0 10_0             [18] &  3_1 11_1 16_1           [ 9] & \infty_0  8_0  8_1        [19] &  \infty_1 14_0 15_1       [17] & \infty_2  22_0  24_1     [25] & \infty_3 11_0 14_1       [ 2] \\   
\infty_4 13_0 17_1        [16] & \infty_5 17_0 22_1       [24] & \infty_6  7_0 13_1        [15] & \infty_7   2_0  9_1       [ 4] & \infty_8  24_0   5_1     [12] & \infty_9  1_0 10_1       [11] \\
\infty_{10} 15_0 25_1     [20] & \infty_{11} 12_0 23_1    [ 7] & \infty_{12} 19_0  4_1     [ 6] & \infty_{13} 16_0  2_1     [21] & \infty_{14} 20_0  7_1    [ 1] & \infty_{15} 18_0  6_1    [ 5] \\
\infty_{16} 23_0 12_1     [ 3] & \infty_{17}  3_0 20_1    [14] & \infty_{18}  9_0  1_1     [23] & \infty_{19} 25_0 18_1     [13] & \infty_{20}  5_0 26_1    [10] & \infty_{21} 26_0 21_1    [ 8] \\
\infty_{22}  4_0  0_1     [26] & \infty_{23} 21_0 19_1    [22] 
\end{array}
\]

\end{example}

\begin{example} For $n=32$: 
\[
\begin{array}{@{}*{6}l@{}}
0_0  4_0  7_0             [22] &  5_1  6_1 11_1           [11] & \infty_0 25_0 25_1        [ 8] &  \infty_1  8_0  9_1       [17] & \infty_2  16_0  18_1     [25] & \infty_3 11_0 14_1       [ 1] \\   
\infty_4 13_0 17_1        [14] & \infty_5 17_0 22_1       [29] & \infty_6  6_0 12_1        [15] & \infty_7   1_0  8_1       [ 3] & \infty_8   2_0  10_1     [32] & \infty_9 10_0 19_1       [13] \\
\infty_{10}  3_0 13_1     [28] & \infty_{11} 18_0 29_1    [18] & \infty_{12} 19_0 31_1     [21] & \infty_{13} 21_0  1_1     [23] & \infty_{14} 22_0  4_1    [26] & \infty_{15} 20_0  3_1    [ 4] \\
\infty_{16} 23_0  7_1     [16] & \infty_{17} 12_0 30_1    [31] & \infty_{18} 14_0  0_1     [ 5] & \infty_{19} 15_0  2_1     [ 2] & \infty_{20}  5_0 26_1    [27] & \infty_{21} 27_0 16_1    [20] \\
\infty_{22}  9_0 32_1     [ 7] & \infty_{23} 24_0 15_1    [ 6] & \infty_{24} 29_0 21_1     [24] & \infty_{25} 30_0 23_1     [12] & \infty_{26} 26_0 20_1    [ 9] & \infty_{27} 32_0 27_1    [19] \\
\infty_{28} 28_0 24_1     [10] & \infty_{29} 31_0 28_1    [30] 
\end{array}
\]
\end{example}

\begin{example} For $n=38$:
\[
\begin{array}{@{}*{6}l@{}}
0_0  6_0  7_0             [26] &  8_1 14_1 24_1           [13] & \infty_0 20_0 20_1        [20] &  \infty_1 24_0 25_1       [ 1] & \infty_2  17_0  19_1     [ 4] & \infty_3 26_0 29_1       [ 2] \\   
\infty_4 19_0 23_1        [23] & \infty_5 23_0 28_1       [30] & \infty_6 25_0 31_1        [19] & \infty_7   2_0  9_1       [29] & \infty_8   3_0  12_1     [27] & \infty_9  5_0 15_1       [ 7] \\
\infty_{10} 10_0 21_1     [31] & \infty_{11} 18_0 30_1    [25] & \infty_{12} 29_0  3_1     [ 9] & \infty_{13} 35_0 10_1     [24] & \infty_{14} 28_0  4_1    [28] & \infty_{15} 30_0  7_1    [17] \\
\infty_{16} 33_0 11_1     [22] & \infty_{17} 14_0 32_1    [21] & \infty_{18} 37_0 17_1     [12] & \infty_{19} 32_0 13_1     [36] & \infty_{20} 34_0 16_1    [32] & \infty_{21} 12_0 34_1    [33] \\
\infty_{22} 22_0  6_1     [14] & \infty_{23} 16_0  1_1    [34] & \infty_{24} 11_0 36_1     [11] & \infty_{25} 13_0  0_1     [ 5] & \infty_{26} 38_0 26_1    [38] & \infty_{27}  9_0 37_1    [ 6] \\
\infty_{28} 15_0  5_1     [37] & \infty_{29}  8_0 38_1    [16] & \infty_{30}  1_0 33_1     [ 8] & \infty_{31} 27_0 22_1     [35] & \infty_{32} 31_0 27_1    [ 3] & \infty_{33} 21_0 18_1    [18] \\
\infty_{34}  4_0  2_1     [15] & \infty_{35} 36_0 35_1    [10] 
\end{array}
\]

\end{example}

\begin{example} For $n=41$:
\[
\begin{array}{@{}*{6}l@{}}
0_0 24_0 32_0             [28] & 10_1 11_1 13_1           [14] & \infty_0 26_0 26_1        [ 8] &  \infty_1  8_0  9_1       [ 1] & \infty_2  16_0  18_1     [ 5] & \infty_3 29_0 32_1       [10] \\   
\infty_4 18_0 22_1        [15] & \infty_5 22_0 27_1       [36] & \infty_6 17_0 23_1        [33] & \infty_7  35_0  0_1       [11] & \infty_8   4_0  12_1     [31] & \infty_9  5_0 14_1       [19] \\
\infty_{10}  6_0 16_1     [32] & \infty_{11}  9_0 20_1    [39] & \infty_{12}  3_0 15_1     [17] & \infty_{13} 11_0 24_1     [20] & \infty_{14} 33_0  5_1    [35] & \infty_{15}  2_0 17_1    [41] \\
\infty_{16} 21_0 37_1     [34] & \infty_{17} 31_0  6_1    [40] & \infty_{18} 15_0 33_1     [29] & \infty_{19} 30_0  7_1     [24] & \infty_{20} 25_0  3_1    [16] & \infty_{21} 19_0 40_1    [30] \\
\infty_{22} 14_0 36_1     [13] & \infty_{23} 23_0  4_1    [22] & \infty_{24} 20_0  2_1     [37] & \infty_{25} 36_0 19_1     [38] & \infty_{26} 1_0 29_1     [18] & \infty_{27} 10_0 39_1    [12] \\
\infty_{28} 12_0  1_1     [ 2] & \infty_{29} 38_0 28_1    [27] & \infty_{30} 40_0 31_1     [ 7] & \infty_{31}  7_0 41_1     [23] & \infty_{32} 41_0 34_1    [26] & \infty_{33} 27_0 21_1    [ 9] \\
\infty_{34} 13_0  8_1     [ 4] & \infty_{35} 34_0 30_1    [ 6] & \infty_{36} 28_0 25_1     [25] & \infty_{37} 37_0 35_1     [ 0] & \infty_{38} 39_0 38_1    [ 3] 
\end{array}
\]

\end{example}

\begin{example} For $n=44$:
\[
\begin{array}{@{}*{6}l@{}}
0_0  2_0 10_0             [30] & 16_1 30_1 36_1           [15] & \infty_0 11_0 11_1        [28] &  \infty_1 25_0 26_1       [31] & \infty_2  26_0  28_1     [ 2] & \infty_3 21_0 24_1        [5] \\   
\infty_4 16_0 20_1        [ 7] & \infty_5 27_0 32_1       [17] & \infty_6 29_0 35_1        [32] & \infty_7   1_0  8_1       [35] & \infty_8   4_0  12_1     [23] & \infty_9  8_0 17_1       [43] \\
\infty_{10} 19_0 29_1     [33] & \infty_{11} 33_0 44_1    [22] & \infty_{12} 30_0 42_1     [16] & \infty_{13} 36_0  4_1     [12] & \infty_{14} 38_0  7_1    [29] & \infty_{15} 39_0  9_1    [10] \\
\infty_{16} 34_0  5_1     [ 4] & \infty_{17} 37_0 10_1    [27] & \infty_{18} 32_0  6_1     [42] & \infty_{19} 20_0 40_1     [13] & \infty_{20} 42_0 18_1    [24] & \infty_{21} 15_0 37_1    [26] \\
\infty_{22}  9_0 33_1     [38] & \infty_{23} 14_0 39_1    [39] & \infty_{24} 40_0 21_1     [19] & \infty_{25}  7_0 34_1     [36] & \infty_{26} 3_0 31_1     [21] & \infty_{27} 18_0  2_1    [44] \\
\infty_{28} 41_0 27_1     [41] & \infty_{29}  6_0 38_1    [ 3] & \infty_{30} 13_0  1_1     [37] & \infty_{31} 24_0 13_1     [34] & \infty_{32} 35_0 25_1    [25] & \infty_{33} 12_0  3_1    [ 8] \\
\infty_{34} 31_0 23_1     [11] & \infty_{35} 22_0 15_1    [ 9] & \infty_{36} 28_0 22_1     [ 6] & \infty_{37}  5_0  0_1     [20] & \infty_{38} 23_0 19_1    [40] & \infty_{39} 17_0 14_1    [18] \\
\infty_{40} 43_0 41_1     [14] & \infty_{41} 44_0 43_1    [ 1]
\end{array}
\]

\end{example}

  Now the intransitive starters and adders:

\begin{example} For $n=10$:
\[
\begin{array}{@{}*{6}l@{}}
0_2  6_2  7_2             [ 0] &  7_0  8_0  3_2           [ 8] &   5_1     6_1  1_2        [ 2] &   9_0      6_0  8_2       [ 4] &  3_1 0_1   2_2   [ 6] & 3_0 1_0 4_2       [ 1] \\   
4_1  2_1  5_2             [ 9] &  4_0  0_0  8_1           [ 7] &   5_0     1_1  7_1        [ 3] &   2_0      9_1  9_2       [ R] &  9_0 2_1   9_2   [ C]   
\end{array}
\]

\end{example}

\begin{example} For $n=12$:
\[
\begin{array}{@{}*{6}l@{}}
3_2  6_2 10_2             [ 0] &  6_0 10_0  9_0           [ 4] &  10_1     2_1  1_1        [ 8] &   5_0      0_1  9_2       [ 2] &  2_0 7_1  11_2   [10] &  7_0 3_1 1_2       [ 1] \\   
4_0  8_1  2_2             [11] & 11_0  1_0  0_2           [ 5] &   4_1     6_1  5_2        [ 7] &   8_0      9_1  4_2       [ 3] &  0_0 11_1  7_2   [ 9] &  3_0 5_1 8_2       [ R] \\   
5_0  3_1  8_2             [ C]  
\end{array}
\]

\end{example}

\begin{example} For $n=16$:
\[
\begin{array}{@{}*{6}l@{}}
6_2 12_2 15_2             [ 0] &  0_0  2_0  7_0           [10] &  10_1    12_1  1_1        [ 6] &   4_0     10_0  5_2       [ 4] &  8_1 14_1  9_2   [12] & 8_0 9_1 1_2       [13] \\   
6_0  5_1 14_2             [ 3] &  3_0  6_1  8_2           [15] &   5_0     2_1  7_2        [ 1] &  12_0      3_1  0_2       [11] & 14_0 7_1  11_2   [ 5] & 13_0  1_0  4_2    [14] \\   
11_1 15_1  2_2            [ 2] & 11_0  0_1 10_2           [ 9] &   9_0     4_1  3_2        [ 7] &  15_0     13_1 13_2       [ R] & 13_0 15_1 13_2   [ C]   
\end{array}
\]

\end{example}

\begin{example} For $n=18$:
\[
\begin{array}{@{}*{6}l@{}}
1_2  2_2 12_2             [ 0] &  0_0  8_0  7_0           [12] &  12_1     2_1  1_1        [ 6] &   4_0     10_0  3_2       [ 4] & 8_1 14_1   7_2   [14] & 12_0 15_1 13_2     [ 1] \\   
16_0 13_1 14_2            [17] &  1_0  0_1  5_2           [ 5] &   5_0     6_1 10_2        [13] &  14_0      3_1  6_2       [ 3] & 6_0 17_1   9_2   [15] &  3_0 5_1 0_2       [ 8] \\   
13_0 11_1 8_2             [10] &  9_0 11_0 17_2           [16] &   7_1     9_1 15_2        [ 2] &  17_0      4_1 11_2       [11] & 15_0 10_1  4_2   [ 7] &  2_0 16_1 16_2     [ R] \\   
16_0  2_1  16_2           [ C]  
\end{array}
\]

\end{example}

\begin{example} For $n=22$:
\[
\begin{array}{@{}*{6}l@{}}
0_2  7_2 13_2             [ 0] &  0_0  4_0  7_0           [16] &  16_1    20_1  1_1        [ 6] &   8_0     14_0 11_2       [10] & 18_1  2_1 21_2   [12] & 10_0 11_1   4_2   [ 5] \\   
16_0 15_1  9_2            [17] & 19_0  0_1  6_2           [ 9] &   9_0     6_1 15_2        [13] &   2_0      9_1 14_2       [ 3] & 12_0 5_1  17_2   [19] &  1_0 19_1 18_2    [ 2] \\   
21_0  3_1 20_2            [20] & 11_0 21_1 12_2           [18] &  17_0     7_1  8_2        [ 4] &  20_0      3_0  5_2       [14] & 12_1 17_1  19_2  [ 8] & 15_0 13_1  1_2    [15] \\   
 6_0  8_1 16_2            [ 7] &  5_0 14_1  3_2           [21] &  13_0     4_1  2_2        [ 1] &  18_0     10_1 10_2       [ R] & 10_0 18_1 10_2   [ C]   
\end{array}
\]

\end{example}

\begin{example} For $n=28$:
\[
\begin{array}{@{}*{6}l@{}}
4_2 12_2 17_2             [ 0] &  0_0  8_0  5_0           [16] &  16_1    24_1 21_1        [12] &  20_0     26_1 23_2       [ 8] &  6_0  0_1  3_2   [20] & 22_0  5_1  16_2   [19] \\   
24_0 13_1  7_2            [ 9] &  9_0 12_1 27_2           [23] &   7_0     4_1 22_2        [ 5] &  10_0     17_1 18_2       [ 1] & 18_0 11_1 19_2   [27] & 13_0 17_0  2_2    [18] \\   
 3_1  7_1 20_2            [10] & 15_0 25_1  1_2           [ 4] &   1_0    19_1  5_2        [24] &  25_0      2_1 21_2       [21] & 23_0 18_1  14_2  [ 7] &  3_0 15_1 24_2    [17] \\   
 4_0 20_1 13_2            [11] & 21_0  2_0  9_2           [ 6] &  27_1     8_1 15_2        [22] &  11_0     12_0 10_2       [26] &  9_1 10_1   8_2  [ 2] & 27_0  1_1 11_2    [15] \\   
16_0 14_1 26_2            [13] & 26_0 22_1  0_2           [25] &  19_0    23_1 25_2        [ 3] &  14_0      6_1  6_2     [ R] &  6_0 14_1  6_2   [ C]   
\end{array}
\]

\end{example}

\begin{example} For $n=30$:
\[
\begin{array}{@{}*{6}l@{}}
3_2  7_2 10_2             [ 0] &  0_0  8_0  5_0           [18] &  18_1    26_1 23_1        [12] &  18_0     24_1 22_2       [28] & 22_0 16_1  20_2   [ 2] & 28_0  9_1 15_2     [27] \\   
 6_0 25_1 12_2             [ 3] & 27_0  0_1  2_2           [23] &  23_0    20_1 25_2        [ 7] &  10_0     17_1 11_2       [17] &  4_0 27_1  28_2  [13] &  9_0 13_0 21_2     [ 6] \\   
15_1 19_1 27_2             [24] & 19_0 29_1 26_2           [22] &  21_0    11_1 18_2        [ 8] &   7_0     12_1 23_2       [21] &  3_0 28_1  14_2  [ 9] &  1_0 13_1 16_2     [ 1] \\   
14_0  2_1 17_2             [29] & 15_0 24_0  4_2           [20] &   5_1    14_1 24_2        [10] &  25_0     26_0  9_2       [26] & 21_1 22_1   5_2  [ 4] & 29_0  1_1 19_2     [11] \\   
12_0 10_1 0_2             [19] &  2_0  6_1  1_2           [ 5] &  11_0     7_1  6_2        [25] &  20_0      3_1 29_2       [14] & 17_0  4_1 13_2   [16] & 16_0  8_1  8_2     [ R] \\   
 8_0 16_1   8_2         [ C]  
\end{array}
\]

\end{example}

\begin{example} For $n=36$:
\[
\begin{array}{@{}*{6}l@{}}
6_2 18_2 35_2             [ 0] &  0_0  4_0  7_0            [26] &  26_1    30_1 33_1        [10] &  34_0      4_1  3_2       [14] & 18_0 12_1  17_2  [22] & 32_0  7_1 33_2     [35] \\   
 6_0 31_1 32_2             [ 1] & 35_0  2_1 15_2           [25] &  27_0    24_1  4_2        [11] &  28_0     35_1 19_2       [31] & 30_0 23_1  14_2  [ 5] & 21_0 25_1 28_2     [ 8] \\   
33_0 29_1  0_2             [28] & 25_0  1_1 31_2           [30] &  31_0    19_1 25_2        [ 6] &  15_0     20_1  2_2       [27] & 11_0  6_1  29_2  [ 9] &  3_0 13_1  1_2     [ 7] \\   
20_0 10_1  8_2             [29] & 19_0 28_1 21_2           [20] &  12_0     3_1  5_2        [16] &  17_0     18_1 13_2       [34] & 16_0 15_1  11_2  [ 2] & 13_0 27_1 24_2     [19] \\   
10_0 32_1  7_2             [17] &  9_0 22_0 34_2           [12] &  21_1    34_1 10_2        [24] &  29_0      2_0 12_2       [15] &  8_1 17_1  27_2  [21] & 26_0  9_1 30_2     [32] \\   
 5_0 22_1 26_2             [ 4] &  1_0 23_0  9_2           [13] &  14_1     0_1 22_2        [23] &   8_0     14_0 23_2       [33] &  5_1 11_1  20_2  [ 3] & 24_0 16_1 16_2     [ R] \\   
16_0 24_1 16_2         [ C]  
\end{array}
\]

\end{example}

\begin{example} For $n=46$:
\[
\begin{array}{@{}*{6}l@{}}
4_2  17_2 35_2            [ 0] &  0_0  4_0  7_0           [36] &  36_1    40_1 43_1        [10] &  44_0      4_1 45_2       [22] & 26_0 20_1 21_2   [24] & 40_0  5_1  44_2   [25] \\   
30_0 19_1 23_2            [21] & 45_0  2_1  1_2           [45] &   1_0    44_1  0_2        [ 1] &  42_0      3_1 11_2       [ 3] &  6_0 45_1 14_2   [43] & 37_0 41_1 27_2    [ 2] \\   
43_0 39_1 29_2            [44] & 21_0 33_1 12_2           [16] &   3_0    37_1 28_2        [30] &  17_0     22_1  5_2       [11] & 33_0 28_1  16_2  [35] & 13_0 23_1 20_2    [39] \\   
16_0  6_1 13_2            [ 7] & 25_0 34_1  7_2           [32] &  20_0    11_1 39_2        [14] &  31_0     32_1 37_2       [28] & 14_0 13_1  19_2  [18] & 41_0  9_1 30_2    [13] \\   
22_0  8_1 43_2            [33] & 15_0 32_0 42_2           [ 6] &  21_1    38_1  2_2        [40] &  19_0     24_0 36_2       [ 5] & 24_1 29_1 41_2   [41] & 38_0  7_1 18_2    [ 4] \\   
11_0 42_1 22_2            [42] & 27_0  9_0 40_2           [37] &  18_1     0_1 31_2        [ 9] &   8_0     18_0 38_2       [17] & 25_1 35_1   9_2  [29] & 36_0 12_1  6_2    [27] \\   
39_0 17_1 33_2            [19] & 10_0 31_1 34_2           [20] &   5_0    30_1  8_2        [26] &  28_0      1_1 15_2       [34] & 35_0 16_1   3_2  [12] &  2_0 15_1 24_2    [ 8] \\   
23_0 10_1 32_2            [38] & 12_0 14_1 10_2           [15] &  29_0    27_1 25_2        [31] &  34_0     26_1 26_2       [ R] & 26_0 34_1  26_2  [ C]   
\end{array}
\]

\end{example}

\section{Starters and adders for small \mbox{\boldmath $\ghd(n+2,3n)$}}
\label{starters and adders}

%\subsection{\boldmath $n=8$}

First we give those obtained by transitive starters and adders:

\begin{example} For $n=8$: 
\[
\begin{array}{@{}*{8}l@{}}
0_{1} 6_{1} 7_{1}  [2] & 4_{1} 2_{1} 7_{0}  [3] & 6_{0} 8_{0} 8_{1}  [8] & 0_{0} 1_{0} 4_{0}  [7] & \infty_0 2_{0} 3_{1}  [1] & \infty_1 5_{0} 9_{1}  [4] & \infty_2 3_{0} 1_{1}  [9] & \infty_3 9_{0} 5_{1}  [6] 
\end{array}
\] \end{example}

%\subsection{\boldmath $n=9$}

\begin{example} For $n=9$: 
\[
\begin{array}{@{}*{7}l@{}}
4_{0} 5_{1} 3_{1}  [8] & 10_{0} 2_{0} 3_{0}  [4] & 6_{0} 8_{0} 6_{1}  [3] & 9_{1} 10_{1} 2_{1}  [6] & \infty_0 0_{0} 4_{1}  [2] & \infty_1 7_{0} 1_{1}  [9] & \infty_2 9_{0} 0_{1}  [1] \\
\infty_3 5_{0} 8_{1}  [10] & \infty_4 1_{0} 7_{1}  [7] & 
\end{array}
\] \end{example}

%\subsection{\boldmath $n=10$}

\begin{example} For $n=10$: 
\[
\begin{array}{@{}*{7}l@{}}
3_{1} 10_{1} 11_{1} [3]   & 3_{0} 4_{0} 6_{0} [10]   & 1_{0} 8_{0} 4_{1} [7]   & 7_{0} 5_{1} 8_{1} [4]   & \infty_{0} 0_{0} 0_{1} [5]   & \infty_{1} 11_{0} 6_{1} [1]   & \infty_{2} 10_{0} 9_{1} [11] \\
\infty_{3} 5_{0} 2_{1} [2]  & \infty_{4} 2_{0} 7_{1}  [8]   & \infty_{5} 9_{0} 1_{1} [9]
\end{array}
\] \end{example}

%\subsection{\boldmath$n=11$}

\begin{example} For $n=11$: 
\[
\begin{array}{lllllll}
3_{1} 11_{0} 0_{0} [7] & 9_{0} 2_{0} 5_{0} [4] & 10_{1} 11_{1} 2_{1} [6] & 4_{1} 10_{0} 6_{1}[1] & \infty_0 3_{0} 0_{1}[0] & \infty_1 12_{0} 5_{1}  [9] & \infty_2 8_{0} 9_{1} [2] \\
\infty_3 6_{0} 1_{1} [8] & \infty_4 4_{0} 8_{1}  [11] &  \infty_5 1_{0} 12_{1} [3] & \infty_6 7_{0} 7_{1} [5]
\end{array}
\] \end{example}

%\subsection{\boldmath $n=12$}

\begin{example} For $n=12$: 
\[
\begin{array}{@{}*{7}l@{}}
3_{1} 7_{0} 2_{0}  [2] & 8_{1} 0_{1} 10_{1}  [6] & 1_{1} 6_{1} 6_{0}  [11] & 1_{0} 3_{0} 4_{0}  [9] & \infty_{0} 11_{0} 13_{1}  [5] & \infty_{1} 0_{0} 5_{1}  [8] \\
 \infty_{2} 5_{0} 9_{1}  [1] & \infty_{3} 8_{0} 2_{1}  [13] & \infty_{4} 10_{0} 7_{1}  [4] & \infty_{5} 12_{0} 4_{1}  [3] & \infty_{6} 13_{0} 11_{1}  [12] & \infty_{7} 9_{0} 12_{1}  [10] & 
\end{array}
\] \end{example}

%\subsection{\boldmath $n=13$}

\begin{example} For $n=13$: 
\[
\begin{array}{lllllll}
0_{0} 4_{1} 8_{1}  [4] & 6_{0} 7_{0} 9_{0}  [8] & 11_{0} 1_{0} 2_{1} [7] & 5_{1} 6_{1}, 11_{1} [9] & \infty_0,13_{0} 7_{1} [0] & \infty_1 8_{0} 3_{1}  [14] \\
 \infty_2 4_{0} 1_{1}  [5] &  \infty_3 14_{0}, 13_{1}  [12] & \infty_4 12_{0}  0_{1}  [13] & \infty_5 10_{0} 12_{1} [6] & \infty_6 5_{0} 10_{1}  [1] & \infty_7 3_{0} 14_{1}  [2] \\
 \infty_8 2_{0} 9_{1}  [10]
\end{array}
\] \end{example}

%\subsection{\boldmath $n=14$}

\begin{example} For $n=14$: 
\[
\begin{array}{@{}*{7}l@{}}
6_{0} 9_{1} 8_{1}  [5] & 9_{0} 13_{0} 0_{0}  [7] & 10_{0} 5_{1} 11_{0}  [3] & 0_{1} 4_{1} 6_{1}  [1] & \infty_0 3_{0} 12_{1}  [6] & \infty_1 12_{0} 10_{1}  [9] \\
 \infty_2 8_{0} 14_{1}  [14] & \infty_3 15_{0} 3_{1}  [13] & \infty_4 4_{0} 11_{1}  [4] & \infty_5 7_{0} 15_{1}  [11] & \infty_6 1_{0} 2_{1}  [2] & \infty_7 2_{0} 7_{1}  [15] \\
 \infty_8 5_{0} 1_{1}  [10] & \infty_9 14_{0} 13_{1}  [12]
\end{array}
\] \end{example}

%\subsection{\boldmath $n=15$}

\begin{example} For $n=15$: 
\[
\begin{array}{@{}*{7}l@{}}
10_{1} 8_{1} 16_{1}  [7] & 16_{0} 4_{0} 13_{1}  [11] & 7_{0} 6_{1} 5_{1}  [4] & 2_{0} 13_{0} 0_{0}  [16] & \infty_0 14_{0} 4_{1}  [0] &  \infty_1 15_{0}  1_{1}  [15] \\
 \infty_2 10_{0}, 11_{1}  [9] & \infty_3 6_{0} 0_{1}  [14] & \infty_4 12_{0} 7_{1}  [12] &  \infty_5 9_{0} 14_{1}  [8] & \infty_6 1_{0} 3_{1}  [5] & \infty_7 3_{0} 9_{1}  [2] \\
 \infty_8 8_{0} 12_{1} [1] & \infty_{9} 5_{0} 15_{1}  [3]  & \infty_{10} 11_{0} 2_{1}  [10] 
\end{array}
\] \end{example}

%\subsection{\boldmath $n=16$} 

\begin{example} For $n=16$: 
\[
\begin{array}{lllllll}
0_{0} 16_{0} 4_{0}  [10] & 17_{1} 2_{1} 6_{0}  [14] & 12_{1} 2_{0} 17_{0}  [5] & 13_{1} 0_{1} 1_{1}  [7] & \infty_{0} 14_{0} 14_{1}  [1] & \infty_{1} 15_{0} 5_{1}  [4] \\
 \infty_{2} 13_{0} 4_{1}  [17] &  \infty_{3} 11_{0} 15_{1}  [13] & \infty_{4} 9_{0} 3_{1}  [8] & \infty_{5} 10_{0} 16_{1} [6] & \infty_{6} 3_{0} 10_{1}  [2] & \infty_{7} 12_{0} 9_{1}  [15] \\
 \infty_{8} 7_{0} 8_{1}  [11] & \infty_{9} 8_{0} 11_{1}  [3] & \infty_{10} 1_{0} 6_{1}  [12] & \infty_{11} 5_{0} 7_{1}  [16] 
\end{array}
\] \end{example}

%\subsection{\boldmath $n=17$}

\begin{example} For $n=17$: 
\[
\begin{array}{lllllll}
11_{1} 3_{1} 7_{0} [2] & 10_{0} 5_{1} 14_{0} [5] & 0_{1} 10_{1} 15_{1} [16] & 11_{0} 1_{0} 3_{0} [15] & \infty_0 2_{0} 8_{1} [0] & \infty_1 17_{0} 9_{1} [8] \\
 \infty_2 9_{0} 12_{1} [11] & \infty_3 4_{0} 4_{1} [10] & \infty_4 0_{0} 18_{1} [12] & \infty_5 18_{0} 1_{1} [18] & \infty_6 12_{0} 2_{1} [1] & \infty_7 15_{0} 16_{1} [9] \\
 \infty_8 16_{0} 13_{1} [7] & \infty_{9} 6_{0} 14_{1} [4] & \infty_{10} 5_{0} 17_{1} [17] & \infty_{11} 13_{0} 7_{1} [14] & \infty_{12} 8_{0} 6_{1} [3] 
\end{array}
\] \end{example}

%\subsection{\boldmath $n=18$}

\begin{example} For $n=18$: 
\[
\begin{array}{lllllll}
17_{0} 18_{1} 17_{1}  [7] & 1_{0} 15_{0} 10_{0}  [17] & 19_{0} 16_{0} 1_{1}  [1] & 0_{1} 2_{1} 16_{1}  [15] & \infty_0 9_{0} 5_{1}  [4] & \infty_1 8_{0} 6_{1}  [13] \\
 \infty_2 18_{0} 11_{1}  [12] & \infty_3 6_{0} 3_{1}  [5] & \infty_4 5_{0} 4_{1}  [3] & \infty_5 14_{0} 9_{1}  [11] & \infty_6 4_{0} 13_{1}  [19] & \infty_7 2_{0} 12_{1}  [14] \\
 \infty_8 13_{0} 7_{1}  [9] & \infty_9 3_{0} 14_{1}  [16] & \infty_{10} 7_{0} 10_{1}  [8] & \infty_{11} 11_{0} 15_{1}  [18] & \infty_{12} 12_{0} 19_{1}  [2] & \infty_{13} 0_{0} 8_{1}  [6] 
\end{array}
\] \end{example}

%\subsection{\boldmath $n=19$}

\begin{example} For $n=19$: 
\[
\begin{array}{llllll}
0_{0} 1_{0} 5_{0} [0] & 16_{1} 3_{1} 7_{0} [1] & 18_{0} 16_{0} 8_{1} [16] & 6_{1} 7_{1} 9_{1} [12] & \infty_0 20_{0} 20_{1} [3] & \infty_1 12_{0} 14_{1} [2] \\
\infty_2 17_{0} 4_{1} [10] & \infty_3 9_{0} 2_{1} [6] & \infty_4 10_{0} 11_{1} [9] & \infty_5 19_{0} 17_{1} [5] & \infty_6 14_{0} 0_{1} [11] & \infty_7 4_{0} 1_{1} [14] \\
 \infty_8 3_{0} 18_{1} [13] & \infty_{9} 2_{0} 12_{1} [15] & \infty_{10} 13_{0} 19_{1} [7] & \infty_{11} 15_{0} 10_{1} [18] & \infty_{12} 6_{0} 5_{1} [4] & \infty_{13} 11_{0} 15_{1} [19] \\ 
\infty_{14} 8_{0} 13_{1} [20] 
\end{array}
\] \end{example}

%\subsection{\boldmath $n=20$}

\begin{example} For $n=20$: 
\[
\begin{array}{@{}*{6}l@{}}
1_{1} 2_{1} 9_{0}  [15] & 16_{1} 5_{0} 10_{0}  [21] & 0_{0} 20_{0} 4_{0}  [3] & 7_{1} 17_{1} 21_{1}  [5] & \infty_0 15_{0} 18_{1}  [1] & \infty_1 18_{0} 15_{1}  [18] \\
\infty_2 11_{0} 5_{1}  [4] & \infty_3 2_{0} 9_{1}  [9] & \infty_4 13_{0} 4_{1}  [19] & \infty_5 1_{0} 10_{1}  [20] & \infty_6 16_{0} 20_{1}  [12] & \infty_7 14_{0} 13_{1}  [16] \\
\infty_8 7_{0} 3_{1}  [10] & \infty_9 12_{0} 0_{1}  [6] & \infty_{10} 6_{0} 14_{1}  [7] & \infty_{11} 3_{0} 8_{1}  [17] & \infty_{12} 17_{0} 12_{1}  [2] & \infty_{13} 8_{0} 6_{1}  [14] \\
\infty_{14} 19_{0} 19_{1}  [8] & \infty_{15} 21_{0} 11_{1}  [13] & 
\end{array}
\] \end{example}

%\subsection{\boldmath $n=21$}

 \begin{example} For $n=21$: 
\[ 
\begin{array}{llllll} 
22_{0} 5_{0} 7_{0} [13] & 21_{0} 14_{0} 3_{1} [5] & 20_{1}  8_{1} 13_{0} [15] & 12_{1}  4_{1}  5_{1} [2] & \infty_0 11_{0}  9_{1} [0] & \infty_1 15_{0} 17_{1} [1] \\
 \infty_2 0_{0} 13_{1} [4] & \infty_3 12_{0}  0_{1} [11] & \infty_4 3_{0} 22_{1} [3] & \infty_5 17_{0}  2_{1} [8] & \infty_6 6_{0} 15_{1} [9] & \infty_7 8_{0} 11_{1} [16] \\
\infty_8 19_{0} 18_{1} [21] & \infty_{9} 2_{0} 19_{1} [7] & \infty_{10} 4_{0} 14_{1} [6] & \infty_{11} 10_{0} 10_{1} [12] & \infty_{12} 1_{0} 16_{1} [20] & \infty_{13} 16_{0}  7_{1} [14] \\
\infty_{14} 20_{0} 21_{1} [17] & \infty_{15} 9_{0}  6_{1} [22] & \infty_{16} 18_{0}  1_{1} [18] 
\end{array} 
\] \end{example}

%\subsection{\boldmath $n=22$}

\begin{example} For $n=22$: 
\[
\begin{array}{@{}*{6}l@{}}
13_{1} 9_{1} 20_{0}  [3] & 10_{0} 23_{0} 7_{1}  [15] & 17_{1} 22_{1} 19_{1}  [9] & 19_{0} 0_{0} 1_{0}  [21] & \infty_0 3_{0} 18_{1}  [1] & \infty_1 4_{0} 6_{1}  [2] \\
\infty_2 11_{0} 10_{1}  [7] & \infty_3 14_{0} 4_{1}  [5] & \infty_4 13_{0} 5_{1}  [20] & \infty_5 17_{0} 20_{1}  [14] & \infty_6 12_{0} 12_{1}  [17] & \infty_7 15_{0} 16_{1}  [11] \\
\infty_8 22_{0} 2_{1}  [22] & \infty_9 9_{0} 21_{1}  [23] & \infty_{10} 18_{0} 3_{1}  [18] & \infty_{11} 8_{0} 14_{1}  [16] & \infty_{12} 21_{0} 8_{1}  [6] & \infty_{13} 16_{0} 23_{1}  [19] \\
\infty_{14} 6_{0} 11_{1}  [4] & \infty_{15} 7_{0} 1_{1}  [10] & \infty_{16} 2_{0} 0_{1}  [13] & \infty_{17} 5_{0} 15_{1}  [8] & 
\end{array}
\] \end{example}

%\subsection{\boldmath $n=23$}

\begin{example} For $n=23$: 
\[
\begin{array}{@{}*{6}l@{}}
12_{0} 8_{0} 21_{1}  [1] & 19_{1} 17_{0} 20_{0}  [6] & 14_{1} 13_{0} 7_{0}  [14] & 2_{1} 11_{1} 16_{1}  [7] & \infty_0 18_{0} 5_{1}  [0] & \infty_1 6_{0} 12_{1}  [2] \\
\infty_2 3_{0} 22_{1}  [9] & \infty_3 19_{0} 17_{1}  [10] & \infty_4 5_{0} 20_{1}  [17] & \infty_5 22_{0} 13_{1}  [19] & \infty_6 24_{0} 7_{1}  [8] & \infty_7 1_{0} 18_{1}  [24] \\
\infty_8 15_{0} 1_{1}  [15] & \infty_9 0_{0} 10_{1}  [3] & \infty_{10} 21_{0} 24_{1}  [21] & \infty_{11} 10_{0} 15_{1}  [4] & \infty_{12} 9_{0} 4_{1}  [22] & \infty_{13} 11_{0} 8_{1}  [13] \\
\infty_{14} 14_{0} 3_{1}  [5] & \infty_{15} 23_{0} 23_{1}  [12] & \infty_{16} 2_{0} 6_{1}  [18] & \infty_{17} 16_{0} 9_{1}  [20] & \infty_{18} 4_{0} 0_{1}  [11] & 
\end{array}
\] \end{example}

%\subsection{\boldmath $n=24$}

\begin{example} For $n=24$: 
\[
\begin{array}{@{}*{6}l@{}}
4_{0} 24_{0} 23_{0}  [12] & 14_{0} 0_{0} 5_{0}  [14] & 4_{1} 15_{1} 16_{1}  [19] & 8_{1} 12_{1} 17_{1}  [7] & \infty_0 21_{0} 13_{1}  [1] & \infty_1 1_{0} 10_{1}  [2] \\
\infty_2 15_{0} 25_{1}  [5] & \infty_3 13_{0} 21_{1}  [18] & \infty_4 16_{0} 20_{1}  [11] & \infty_5 25_{0} 0_{1} [22] & \infty_6 11_{0} 14_{1}  [23] & \infty_7 12_{0} 19_{1}  [6] \\
\infty_8 20_{0} 6_{1}  [10] & \infty_9 6_{0} 11_{1}  [17] & \infty_{10} 17_{0} 23_{1}  [21] & \infty_{11} 22_{0} 24_{1}  [3] & \infty_{12} 3_{0} 3_{1}  [4] & \infty_{13} 2_{0} 1_{1}  [9] \\
\infty_{14} 8_{0} 22_{1}  [24] & \infty_{15} 7_{0} 18_{1}  [8] & \infty_{16} 18_{0} 7_{1}  [25] & \infty_{17} 10_{0} 5_{1}  [16] & \infty_{18} 9_{0} 2_{1}  [15] & \infty_{19} 19_{0} 9_{1}  [20] 
\end{array}
\] \end{example}

%\subsection {\boldmath $n=25$}

\begin{example} For $n=25$: 
\[
\begin{array}{@{}*{6}l@{}}
15_{1} 3_{1} 5_{1}  [19] & 7_{0} 1_{0} 11_{0}  [9] & 16_{1} 21_{1} 25_{1}  [25] & 14_{0} 9_{0} 16_{0}  [8] & \infty_0 13_{0} 11_{1}  [0] & \infty_1 2_{0} 17_{1}  [1] \\
\infty_2 12_{0} 7_{1}  [2] & \infty_3 5_{0} 23_{1}  [3] & \infty_4 21_{0} 22_{1}  [5] & \infty_5 0_{0} 4_{1}  [6] & \infty_6 25_{0} 9_{1}  [23] & \infty_7 18_{0} 26_{1}  [18] \\
\infty_8 6_{0} 13_{1}  [26] & \infty_9 20_{0} 19_{1}  [14] & \infty_{10} 23_{0} 8_{1}  [22] & \infty_{11} 24_{0} 18_{1}  [7] & \infty_{12} 26_{0} 1_{1}  [12] & \infty_{13} 3_{0} 0_{1}  [16] \\
\infty_{14} 22_{0} 12_{1}  [17] & \infty_{15} 17_{0} 10_{1}  [10] & \infty_{16} 8_{0} 14_{1}  [21] & \infty_{17} 10_{0} 6_{1}  [15] & \infty_{18} 15_{0} 2_{1}  [13] & \infty_{19} 4_{0} 20_{1}  [11] \\
\infty_{20} 19_{0} 24_{1}  [4] & 
\end{array}
\] \end{example}

%\subsection {\boldmath $n=26$}

\begin{example} For $n=26$: 
\[
\begin{array}{@{}*{6}l@{}}
23_{0} 23_{1} 1_{1}  [11] & 8_{0} 9_{0} 20_{0}  [5] & 11_{0} 3_{0} 7_{1}  [15] & 16_{1} 5_{1} 6_{1}  [25] & \infty_0 1_{0} 8_{1}  [1] & \infty_1 2_{0} 21_{1} [2] \\
\infty_2 13_{0} 2_{1} [3] & \infty_3 0_{0} 10_{1}  [9] & \infty_4 27_{0} 11_{1}  [4] & \infty_5 25_{0} 15_{1}  [20] & \infty_6 24_{0} 17_{1}  [12] & \infty_7 12_{0} 26_{1}  [27] \\
\infty_8 16_{0} 18_{1}  [24] & \infty_9 7_{0} 22_{1}  [22] & \infty_{10} 19_{0} 24_{1}  [8] & \infty_{11} 5_{0} 14_{1}  [10] & \infty_{12} 15_{0} 3_{1}  [7] & \infty_{13} 17_{0} 0_{1}  [21] \\
\infty_{14} 26_{0} 25_{1}  [23] & \infty_{15} 10_{0} 4_{1}  [13] & \infty_{16} 14_{0} 27_{1} [19] & \infty_{17} 6_{0} 9_{1}  [18] & \infty_{18} 4_{0} 12_{1}  [16] & \infty_{19} 18_{0} 19_{1}  [17] \\
\infty_{20} 21_{0} 13_{1}  [26] & \infty_{21} 22_{0} 20_{1}  [6] &
\end{array}
\] \end{example}

%\subsection {\boldmath $n=27$}

\begin{example} For $n=27$: 
\[
\begin{array}{@{}*{6}l@{}}
4_{0} 8_{0} 22_{0}  [17] & 11_{0} 19_{0} 22_{1}  [11] & 5_{1} 18_{1} 6_{0}  [14] & 2_{1} 10_{1} 4_{1}  [16] & \infty_0 24_{0} 16_{1}  [0] & \infty_1 14_{0} 27_{1}  [1] \\
\infty_2 25_{0} 14_{1}  [3] & \infty_3 9_{0} 26_{1}  [4] & \infty_4 27_{0} 8_{1}  [2] & \infty_5 10_{0} 19_{1}  [6] & \infty_6 26_{0} 20_{1}  [7] & \infty_7 0_{0} 24_{1}  [27] \\
\infty_8 16_{0} 21_{1}  [21] & \infty_9 18_{0} 9_{1}  [5] & \infty_{10} 20_{0} 17_{1}  [12] & \infty_{11} 23_{0} 25_{1}  [13] & \infty_{12} 12_{0} 13_{1}  [22] & \infty_{13} 2_{0} 0_{1}  [15] \\
\infty_{14} 3_{0} 28_{1}  [8] & \infty_{15} 13_{0} 6_{1}  [28] & \infty_{16} 1_{0} 15_{1}  [25] & \infty_{17} 15_{0} 1_{1}  [23] & \infty_{18} 17_{0} 23_{1}  [18] & \infty_{19} 7_{0} 7_{1}  [24] \\
\infty_{20} 28_{0} 3_{1}  [20] & \infty_{21} 21_{0} 11_{1}  [26] & \infty_{22} 5_{0} 12_{1}  [9] & 
\end{array}
\] \end{example}

%\subsection {\boldmath $n=28$}

\begin{example} For $n=28$: 
\[
\begin{array}{@{}*{6}l@{}}
22_{0} 3_{0} 6_{0}  [26] & 27_{0} 2_{0} 28_{0}  [24] & 10_{1} 18_{1} 13_{1}  [23] & 8_{1} 14_{1} 27_{1}  [17] & \infty_0 12_{0} 25_{1}  [1] & \infty_1 10_{0} 22_{1}  [2] \\
\infty_2 16_{0} 2_{1}  [3] & \infty_3 1_{0} 19_{1}  [4] & \infty_4 5_{0} 5_{1}  [5] & \infty_5 15_{0} 24_{1}  [8] & \infty_6 24_{0} 29_{1}  [13] & \infty_7 4_{0} 23_{1}  [29] \\
\infty_8 25_{0} 12_{1}  [9] & \infty_9 7_{0} 17_{1}  [20] & \infty_{10} 14_{0} 6_{1}  [14] & \infty_{11} 26_{0} 28_{1} [18] & \infty_{12} 13_{0} 16_{1}  [12] & \infty_{13} 0_{0} 7_{1}  [11] \\
\infty_{14} 18_{0} 11_{1}  [28] & \infty_{15} 29_{0} 26_{1}  [21] & \infty_{16} 20_{0} 4_{1}  [25] & \infty_{17} 21_{0} 20_{1}  [10] & \infty_{18} 17_{0} 21_{1}  [22] & \infty_{19} 9_{0} 0_{1}  [27] \\
\infty_{20} 23_{0} 1_{1}  [7] & \infty_{21} 8_{0} 3_{1}  [16] & \infty_{22} 19_{0} 15_{1}  [19] & \infty_{23} 11_{0} 9_{1}  [6] &
\end{array}
\] \end{example}

%\subsection{\boldmath $n=29$}

\begin{example} For $n=29$: 
\[
\begin{array}{@{}*{6}l@{}}
12_{0} 17_{0} 13_{1}  [29] & 9_{0} 10_{0} 0_{0}  [13] & 7_{1} 23_{1} 25_{1}  [7] & 27_{1} 30_{1} 2_{0}  [16] & \infty_{0} 27_{0} 19_{1}  [0] & \infty_{1} 20_{0} 2_{1}  [1] \\
\infty_{2} 15_{0} 22_{1}  [2] & \infty_{3} 26_{0} 5_{1}  [3] & \infty_{4} 3_{0} 18_{1}  [4] & \infty_{5} 28_{0} 15_{1}  [5] & \infty_{6} 22_{0} 26_{1}  [9] & \infty_{7} 21_{0} 29_{1}  [11] \\
\infty_{8} 14_{0} 28_{1}  [10] & \infty_{9} 18_{0} 8_{1}  [17] & \infty_{10} 6_{0} 17_{1}  [30] & \infty_{11} 16_{0} 14_{1}  [14] & \infty_{12} 24_{0} 12_{1}  [19] & \infty_{13} 29_{0} 20_{1}  [21] \\
\infty_{14} 19_{0} 0_{1}  [18] & \infty_{15} 8_{0} 24_{1}  [12] & \infty_{16} 25_{0} 3_{1}  [20] & \infty_{17} 4_{0} 4_{1}  [22] & \infty_{18} 7_{0} 10_{1}  [27] & \infty_{19} 30_{0} 1_{1}  [26] \\
\infty_{20} 23_{0} 9_{1}  [24] & \infty_{21} 11_{0} 16_{1}  [28] & \infty_{22} 5_{0} 11_{1}  [6] & \infty_{23} 1_
{0} 21_{1}  [8] & \infty_{24} 13_{0} 6_{1}  [15] & 
\end{array}
\] \end{example}

%\subsection{\boldmath $n=31$}

\begin{example} For $n=31$: 
\[
\begin{array}{@{}*{6}l@{}}
11_{0} 19_{0} 13_{0}  [28] & 22_{1} 17_{1} 24_{1}  [26] & 4_{0} 24_{0} 27_{0}  [9] & 28_{1} 19_{1} 31_{1}  [11] & \infty_{0} 1_{0} 5_{1}  [0] & \infty_{1} 23_{0} 30_{1}  [1] \\
\infty_{2} 25_{0} 26_{1}  [2] & \infty_{3} 17_{0} 10_{1}  [3] & \infty_{4} 31_{0} 8_{1}  [4] & \infty_{5} 5_{0} 18_{1}  [5] & \infty_{6} 6_{0} 20_{1}  [6] & \infty_{7} 12_{0} 3_{1}  [13] \\
\infty_{8} 28_{0} 13_{1}  [31] & \infty_{9} 8_{0} 25_{1}  [7] & \infty_{10} 14_{0} 11_{1}  [16] & \infty_{11} 3_{0} 12_{1}  [8] & \infty_{12} 20_{0} 23_{1}  [17] & \infty_{13} 9_{0} 14_{1}  [23] \\
\infty_{14} 26_{0} 21_{1}  [12] & \infty_{15} 21_{0} 0_{1}  [29] & \infty_{16} 16_{0} 32_{1}  [15] & \infty_{17} 22_{0} 4_{1}  [30] & \infty_{18} 15_{0} 7_{1}  [14] & \infty_{19} 0_{0} 6_{1}  [18] \\
\infty_{20} 10_{0} 9_{1}  [32] & \infty_{21} 29_{0} 27_{1}  [25] & \infty_{22} 32_{0} 1_{1}  [24] & \infty_{23} 30_{0} 16_{1}  [19] & \infty_{24} 18_{0} 29_{1}  [22] & \infty_{25} 7_{0} 15_{1}  [21] \\
\infty_{26} 2_{0} 2_{1}  [20] & 
\end{array}
\] \end{example}

\begin{example} For $n=32$: 
\[
\begin{array}{@{}*{6}l@{}}
0_0  3_0  11_0  [31] & 12_0 13_0 27_0  [3] & 23_1 24_1 29_1 [7] & 15_1 26_1  5_1  [27] & \infty_0 17_0 17_1 [9] & \infty_1 15_0 16_1 [12] \\
\infty_2 26_0 28_1  [26] & \infty_3 1_0  4_1  [1] & \infty_4 2_0  6_1 [4] & \infty_5 4_0  9_1 [5] & \infty_6 5_0 11_1 [6] & \infty_7 6_0 13_1 [11] \\
\infty_8 10_0 18_1 [22] & \infty_9 16_0 25_1 [13] & \infty_{10} 9_0 19_1 [30] & \infty_{11} 19_0 30_1  [25] & \infty_{12} 20_0 32_1  [2] & \infty_{13} 18_0 31_1  [15] \\
\infty_{14} 21_0  1_1  [32] & \infty_{15} 22_0  3_1  [24] & \infty_{16} 28_0 10_1 [19] & \infty_{17} 25_0  8_1  [10] & \infty_{18} 23_0  7_1  [18] & \infty_{19} 29_0 14_1  [29] \\
\infty_{20} 14_0  0_1  [23] & \infty_{21} 33_0 20_1 [21] & \infty_{22} 24_0 12_1   [33] & \infty_{23} 32_0 21_1    [16] & \infty_{24} 31_0 22_1    [28] & \infty_{25} 7_0  33_1    [14] \\
\infty_{26} 8_0  2_1      [20] & \infty_{27} 30_0 27_1      [8]
\end{array}
\]
\end{example}

\begin{example} For $n=33$:
\[
\begin{array}{@{}*{6}l@{}}
0_0 11_0 32_0          [0]  &  6_0 23_0 25_0        [ 8]  & 16_1 17_1 21_1        [11]  &  7_1 27_1 33_1         [32] & \infty_0 17_0 18_1    [22]  & \infty_1 26_0 28_1      [14] \\
\infty_2 16_0 19_1     [26] & \infty_3  1_0  5_1    [ 1] & \infty_4  3_0   8_1    [ 3]  & \infty_5  4_0 10_1     [ 4] & \infty_6  2_0  9_1    [13]  & \infty_7  5_0 13_1      [21] \\
\infty_8 14_0 23_1     [10] & \infty_9 10_0 20_1    [15] & \infty_{10} 13_0 24_1  [34]  & \infty_{11} 18_0 31_1  [ 5] & \infty_{12} 20_0 34_1 [ 9]  & \infty_{13} 21_0  1_1   [24] \\
\infty_{14} 31_0 12_1  [17] & \infty_{15} 22_0  4_1 [33] & \infty_{16} 28_0 11_1  [ 2]  & \infty_{17} 30_0 14_1  [ 6] & \infty_{18} 15_0  0_1 [19]  & \infty_{19} 29_0 15_1   [23] \\
\infty_{20} 19_0  6_1  [25] & \infty_{21}  7_0 30_1 [31] & \infty_{22}  8_0 32_1  [20]  & \infty_{23} 12_0  2_1  [ 7] & \infty_{24} 34_0 25_1 [28]  & \infty_{25}  9_0  3_1   [12] \\
\infty_{26} 33_0 29_1  [18] & \infty_{27} 24_0 22_1 [29] & \infty_{28} 27_0 26_1  [30] 

\end{array}
\]
 \end{example}

\begin{example} For $n=34$:
\[
\begin{array}{@{}*{6}l@{}}
0_0  1_0  11_0     [35] & 15_0 17_0 32_0      [1] & 1_1 23_1 26_1     [11] & 10_1 14_1 27_1     [25] & \infty_0 18_0 18_1          [20] & \infty_1 12_0 13_1          [12] \\
\infty_2 26_0 28_1          [27] & \infty_3 2_0  5_1             [2] & \infty_4 3_0  7_1             [3] & \infty_5 4_0  9_1             [5] & \infty_6 5_0 11_1            [9] & \infty_7 8_0 15_1          [17] \\
\infty_8 9_0 17_1       [14] & \infty_9 7_0 16_1        [24] & \infty_{10} 10_0 20_1     [16] & \infty_{11} 19_0 30_1     [32] & \infty_{12} 20_0 32_1     [28] & \infty_{13} 21_0 34_1       [7] \\ 
\infty_{14} 22_0  0_1      [15] & \infty_{15} 25_0  4_1      [19] & \infty_{16} 23_0  3_1       [6] & \infty_{17} 27_0  8_1      [22] & \infty_{18} 24_0  6_1      [23] & 
\infty_{19} 29_0 12_1   [10] \\ 
\infty_{20} 13_0 33_1   [21] & \infty_{21} 14_0 35_1   [29] & \infty_{22} 16_0  2_1      [4] & \infty_{23} 35_0 22_1 [31] & \infty_{24} 6_0 31_1     [13] & 
\infty_{25} 31_0 21_1    [26] \\
\infty_{26} 28_0 19_1    [30] & \infty_{27} 33_0 25_1     [8] & \infty_{28} 30_0 24_1   [33] & \infty_{29} 34_0 29_1   [34]
\end{array}
\]
\end{example}

\begin{example} For $n=39$:
\[
\begin{array}{@{}*{6}l@{}}
0_0 11_0 18_0          [0]  & 10_0 22_0 36_0        [17]  &  2_1 27_1 33_1        [ 4]  & 13_1 39_1 40_1         [13] & \infty_0 19_0 19_1    [ 2]  & \infty_1 12_0 14_1      [30] \\
\infty_2 20_0 23_1     [18] & \infty_3 16_0 21_1    [28] & \infty_4  2_0   8_1    [35]  & \infty_5  3_0 10_1     [20] & \infty_6  1_0  9_1    [31]  & \infty_7  6_0 15_1      [40] \\
\infty_8 14_0 24_1     [15] & \infty_9  4_0 16_1    [ 6] & \infty_{10} 25_0 38_1  [10]  & \infty_{11} 32_0  5_1  [11] & \infty_{12} 29_0  3_1 [25]  & \infty_{13} 31_0  6_1   [29] \\
\infty_{14} 35_0 11_1  [22] & \infty_{15} 24_0  1_1 [26] & \infty_{16} 26_0  4_1  [21]  & \infty_{17}  7_0 28_1  [23] & \infty_{18} 13_0 35_1 [ 7]  & \infty_{19} 40_0 22_1   [37] \\
\infty_{20}  8_0 32_1  [14] & \infty_{21}  9_0 34_1 [36] & \infty_{22} 15_0  0_1  [19]  & \infty_{23} 21_0  7_1  [27] & \infty_{24} 33_0 20_1 [16]  & \infty_{25} 37_0 26_1   [32] \\
\infty_{26}  5_0 36_1  [ 9] & \infty_{27} 27_0 18_1 [38] & \infty_{28} 38_0 31_1  [34] & \infty_{29} 23_0 17_1   [33] & \infty_{30} 17_0 12_1 [ 8] & \infty_{31} 34_0 30_1    [24] \\
 \infty_{32} 28_0 25_1 [39] & \infty_{33} 39_0 37_1 [ 1] & \infty_{34} 30_0 29_1  [ 3] 
\end{array}
\]

\end{example}

\begin{example} For $n=44$:
\[
\begin{array}{@{}*{6}l@{}}
0_0  2_0  7_0       [44] & 21_0 24_0 33_0      [2] & 0_1 11_1 17_1      [6] & 10_1 14_1 19_1     [40] & \infty_0 32_0 32_1          [28] & \infty_1 26_0 27_1          [29] \\
\infty_2 18_0 20_1          [45] & \infty_3 1_0  4_1            [18] & \infty_4 3_0  7_1            [36] & \infty_5 4_0  9_1            [16] & \infty_6 6_0 12_1             [9] & \infty_7 8_0 15_1             [5] \\
\infty_8 10_0 18_1          [30] & \infty_9 34_0 43_1          [37] & \infty_{10} 35_0 45_1          [32] & \infty_{11} 13_0 24_1     [15] & \infty_{12} 29_0 41_1     [41] & \infty_{13} 36_0  3_1     [43] \\
\infty_{14} 38_0  6_1     [38] & \infty_{15} 39_0  8_1      [3] & \infty_{16} 31_0  1_1     [31] & \infty_{17} 17_0 34_1      [1] & \infty_{18} 19_0 37_1     [33] & \infty_{19} 20_0 39_1     [25] \\
\infty_{20} 22_0 42_1     [19] & \infty_{21} 23_0 44_1     [14] & \infty_{22} 14_0 36_1     [13] & \infty_{23} 5_0 28_1     [17] & \infty_{24} 9_0 33_1     [20] & \infty_{25} 15_0 40_1    [34] \\
\infty_{26} 45_0 25_1     [12] & \infty_{27} 40_0 21_1       [8] & \infty_{28} 41_0 23_1     [39] & \infty_{29} 44_0 31_1    [10] & \infty_{30} 12_0  2_1       [24] & \infty_{31} 25_0 16_1    [22] \\
 \infty_{32} 43_0 35_1      [7] & \infty_{33} 37_0 30_1    [21] & \infty_{34} 11_0  5_1     [42] & \infty_{35} 27_0 22_1    [11] & \infty_{36} 42_0 38_1    [35] & \infty_{37} 16_0 13_1    [27] \\
\infty_{38} 28_0 26_1      [4] & \infty_{39} 30_0 29_1    [26]
\end{array}
\]

\end{example}

Now the intransitive starters and adders:

\begin{example} For $n= 7$: 
\[
\begin{array}{@{}*{7}l@{}}
 2_{2} 3_{2} 5_{2} [0] & 0_{0} 1_{0}  3_{0} [6] & 6_{1} 0_{1} 2_{1} [1] & 6_{0}  3_{1} 4_{2} [ 2] & 5_0 1_{1} 6_{2} [5] & 4_{0} 5_1  0_{2} [R] &  5_{0} 4_{1} 0_{2} [C] \\
 2_{0}  4_{1} 1_{2} [ R] & 4_0 2_{1} 1_{2} [C]  \\
\end{array}
\] \end{example}

\begin{example} For $n= 9$: 
\[
\begin{array}{@{}*{7}l@{}}
 0_{2} 3_{2} 5_{2} [0] & 6_{0} 4_{0}  8_{2} [5] & 2_{1} 0_{1} 4_{2} [4] & 1_{0}  0_{0} 7_{2} [3] & 4_1 3_{1} 1_{2} [6] & 8_{0} 3_0  7_{1} [7] &  5_{0} 6_{1} 1_{1} [2] \\
 2_{0} 8_{1} 2_{2} [R] & 8_{0} 2_{1} 2_{2} [C] & 7_{0}  5_{1} 6_{2} [R] & 5_0 7_{1} 6_{2} [C]  \\
\end{array}
\] \end{example}

\begin{example} For $n=11$: 
\[
\begin{array}{@{}*{7}l@{}}
 0_{2} 1_{2} 9_{2} [0] & 4_{0} 8_{0}  7_{0} [5] & 9_{1} 2_{1} 1_{1} [6] & 5_{0}  4_{1} 7_{2} [ 9] & 2_0 3_{1} 5_{2} [2] & 1_{0} 6_0  6_{2} [4] & 5_{1} 10_{1} 10_{2} [7] \\
 9_{0} 0_{1} 4_{2} [10] & 10_{0} 8_{1} 3_{2} [1] & 0_{0} 7_{1} 8_{2} [R] & 7_0 0_{1} 8_{2} [C] & 3_{0} 6_1 2_{2} [R] &  6_0 3_{1} 2_{2} [C]  \\
\end{array}
\] \end{example}

\end{appendices}


\begin{thebibliography}{20}

\bibitem{Abel}
R.~J.~R.~Abel.  Existence of five MOLS of orders 18 and 60.  {\em J.\ Combin.\ Des.} {\bf 23} (2015), 135--139.

\bibitem{AbelBennett}
R.~J.~R.~Abel and F.~E.~Bennett.  Existence of 2 SOLS and 2 ISOLS.  {\em Discrete Math.} {\bf 312} (2012), 854--867.

\bibitem{AbelBennettGe}
R.~J.~R.~Abel, F.E.~Bennett and G.~Ge.  The existence of four HMOLS with equal sized holes.  {\em Des.\ Codes Crypt.} {\bf 26} (2002), 7--31.

\bibitem{ACCLWW}
%Abel, Chan, Colbourn, Lamken, Wang and Wang
R.~J.~R.~Abel, N.~Chan, C.~J.~Colbourn, E.~R.~Lamken, C.~Wang and J.~Wang.  Doubly resolvable nearly Kirkman triple systems.  {\em J.\ Combin.\ Des.} {\bf 21} (2013), 342--358.

% Orthogonal
\bibitem{ALW}
R.~J.~R.~Abel, E.~R.~Lamken and J.~Wang. 
A few more Kirkman squares and doubly near resolvable BIBDs with block size 3. {\em Discrete Math.} {\bf 308} (2008), 1102--1123.

\bibitem{ASS}
B.~A.~Anderson, P.~J.~Schellenberg and D.~R.~Stinson. The existence of Howell designs of even side. {\em J.\ Combin. Theory, Series A} {\bf 36} (1984), 23--55.

\bibitem{ArhinThesis}
J.~Arhin. On the construction and structure of SOMAs and related partial linear spaces.  Ph.D.\ thesis, University of London, 2006.

\bibitem{Arhin}
J.~Arhin. Every $\soma(n-2,n)$ is Trojan. {\em Discrete Math.} {\bf 310} (2010), 303--311.

\bibitem{packings} 
R.~F.~Bailey and A.~C.~Burgess. Generalized packing designs. {\em Discrete Math.} {\bf 313} (2013), 1167--1190.

\bibitem{BennettColbournZhu}
F.~E.~Bennett, C.~J.~Colbourn and L.~Zhu.  Existence of three HMOLS of types $h^n$ and $2^n3^1$.  {\em Discrete Math.} {\bf 160} (1996), 49--65.

\bibitem{brickell}
E.~F.~Brickell.  A few results in message authentication.  {\em Congr.\ Numer.} {\bf 43} (1984), 141--154.

\bibitem{BrouwerVanRees}
A.~E.~Brouwer and G.~H.~J.~van~Rees.  More mutually orthogonal Latin squares.  {\em Discrete Math.} {\bf 39} (1982), 263--281.

\bibitem{OCD}
A.~Burgess, P.~Danziger, E.~Mendelsohn and B.~Stevens.  Orthogonally resolvable cycle decompositions.  {\em J.\ Combin.\ Des.} {\bf 23} (2015), 328--351. %, to appear. \texttt{DOI:10.1002/jcd.21404}.

\bibitem{Cameron09} 
P.~J.~Cameron. A generalisation of $t$-designs. {\em Discrete Math.} {\bf 309} (2009), 4835--4842.

\bibitem{Chee2} Y.~M.~Chee, Z.~Cherif, J.-L.~Danger, S.~Guilley, H.~M.~Kiah, J.-L.~Kim, P.~Sol\'e and X.~Zhang.  Multiply constant-weight codes and the reliability of loop physically unclonable functions. {\em IEEE Trans.\ Inform.\ Theory} {\bf 60} (2014), 7026--7034.

\bibitem{CKZZ} Y.~M.~Chee, H.~M.~Kiah, H.~Zhang and X.~Zhang.  Constructions of optimal and near-optimal multiply constant-weight codes, preprint.  \url{arxiv.org/abs/1411.2513}

\bibitem{ChuColbournDukes} W.~Chu, C.~J.~Colbourn and P.~Dukes.  Constructions for permutation codes in powerline communications.  {\em Des.\ Codes Cryptogr.} {\bf 32} (2004), 51--64.

\bibitem{Handbook}
C.~J.~Colbourn and J.~H.~Dinitz, editors.
\newblock {\em The CRC Handbook of Combinatorial Designs},
\newblock 2nd ed. CRC Press, Boca Raton, 2007.

\bibitem{Finland}
C.~J.~Colbourn, P.~Kaski, P.~R.~J.~\"{O}sterg{\aa}rd, D.~A.~Pike and O.~Pottonen.
Nearly Kirkman triple systems of order 18 and Hanani triple systems of order 19.
{\em Discrete Math.} {\bf 311} (2011), 827--834. 

\bibitem{ColbournKloveLing} C.~J.~Colbourn, T.~Kl{\o}ve and A.~C.~H.~Ling.  Permutation arrays for powerline communication and mutually orthogonal Latin squares.  {\em IEEE Trans.\ Inform.\ Theory} {\bf 50} (2004), 1289--1291.

% Orthogonal
\bibitem{CLLM}
C.~J.~Colbourn, E.~R.~Lamken, A.~C.~H.~Ling and W.~H.~Mills.
The existence of Kirkman squares--doubly resolvable $(v, 3, 1)$-BIBDs. 
{\em Des.\ Codes Cryptogr.} {\bf 26} (2002), 169--196.

\bibitem{Curran Vanstone}
D.~G.~Curran and S.~A.~Vanstone. Doubly resolvable designs from generalized Bhaskar Rao designs. {\em Discrete Math.} {\bf 73} (1988), 49--63.

\bibitem{DezaVanstone}
M.~Deza and S.~A.~Vanstone.  Bounds for permutation arrays. {\em J.\ Statist.\ Plann.\ Inference} {\bf 2} (1978), 197--209.

% Orthogonal
\bibitem{blue book}
J.~H.~Dinitz and D.~R.~Stinson. Room squares and related designs. In {\em Contemporary Design Theory: A Collection of Surveys} (ed.\ J.~H.~Dinitz and D.~R.~Stinson), Wiley, New York, 1992, pp. 137--204.

\bibitem{DuAbelWang}  J.~Du, R.~J.~R.~Abel and J.~Wang.  Some new resolvable GDDs with $k=4$ and doubly resolvable GDDs with $k=3$.  {\em Discrete Math.} {\bf 338} (2015), 2105--2118. 

\bibitem{Etzion}  T.~Etzion.  Optimal doubly constant weight codes.  {\em J.\ Combin.\ Des.} {\bf 16} (2008), 137--151.

\bibitem{Fuji-Hara Vanstone}
R.~Fuji-Hara and S.~A.~Vanstone. On the spectrum of doubly resolvable Kirkman systems. {\em Congr. Numer.} {\bf 28} (1980), 399--407.

\bibitem{Fuji-Hara Vanstone TD}
R.~Fuji-Hara and S.~A.~Vanstone. Transversal designs and doubly resolvable designs. {\em Europ.\ J.\ Combin.} {\bf 1} (1980), 219--223. 

\bibitem{Vanstone AG}
R.~Fuji-Hara and S.~A.~Vanstone. The existence of orthogonal resolutions of lines in AG$(n,q)$. {\em J.\ Combin.\ Theory Ser.\ A} {\bf 45} (1987), 139--147. 

\bibitem{Huczynska}  S.\ Huczynska.  Powerline communication and the 36 officers problem.  {\em Phil.\ Trans.\ R.\ Soc.\ A} {\bf 364} (2006), 3199--3214.

\bibitem{Kirkman}
T.~P.~Kirkman. Note on an unanswered prize question. {\em Cambridge and Dublin Math.\ J.} {\bf 5} (1850), 255--262.
%Query VI. {\em Lady's and Gentleman's Diary} (1847), 48.

\bibitem{KotzigRosa74} 
A.~Kotzig and A.~Rosa. Nearly Kirkman systems. {\em Proceedings of the Fifth Southeastern Conference on Combinatorics, Graph Theory and Computing (Florida Atlantic Univ., Boca Raton, Fla., 1974)}, pp. 607--614. {\em Congressus Numerantium}, No. X, Utilitas Math., Winnipeg, 1974.

\bibitem{Lamken95}
E.~R.~Lamken. The existence of doubly resolvable $(v,3,2)$-BIBDs. {\em J. Combin. Theory Ser. A} {\bf 72} (1995), 50--76.

\bibitem{Lamken 09}
E.~R.~Lamken. Designs with mutually orthogonal resolutions and decompositions of edge-colored complete graphs. {\em J.\ Combin.\ Des.} {\bf 17} (2009), 425--447.

\bibitem{MacNeish22} H.~MacNeish. Euler squares. {\em Ann. of Math.\ (2)} {\bf 23} (1922), 221--227.

\bibitem{Moore} E.~H.~Moore.  Tactical Memoranda I--III.  {\em Amer.\ J.\ Math.} {\bf 18} (1896), 264--303.

% Orthogonal
\bibitem{Mathon Vanstone}
R.~Mathon and S.~A.~Vanstone. On the existence of doubly resolvable Kirkman systems and equidistant permutation arrays. {\em Discrete Math.} {\bf 30} (1980), 157--172.

% Orthogonal
\bibitem{Mullin}
R.~C.~Mullin and W.~D.~Wallis. The existence of Room squares. {\em Aequationes Math.} {\bf 1} (1975), 1--7.

\bibitem{PhillipsWallis}
N.~C.~K.~Phillips and W.~D.~Wallis. All solutions to a tournament problem. {\em Congr.\ Numer.} {\bf 114} (1996), 193--196.

\bibitem{Room} 
T.~G.~Room. A new type of magic square. {\em Math. Gaz.} {\bf 39} (1955), 307.

\bibitem{Rosa}
A.~Rosa. Generalized Howell designs. In Second International Conference on Combinatorial Mathematics, {\em Ann.\ N.\ Y.\ Acad.\ Sci.} {\bf 319} (1979), 484--489.

\bibitem{Rosa Vanstone}
A.~Rosa and S.~A.~Vanstone. Starter-adder techniques for Kirkman squares and Kirkman cubes of small sides. {\em Ars Combin.} {\bf 14} (1982), 199--212.

\bibitem{Smith77}
P.~Smith. A doubly divisible nearly Kirkman system. {\em Discrete Math.} {\bf 18} (1977), 93--96. 

\bibitem{Soicher}
L.~H.~Soicher. On the structure and classification of SOMAs: generalizations of mutually orthogonal Latin squares. {\em Electronic J. Combin.} {\bf 6(1)} (1999), R32 (15 pages).

\bibitem{Stinson}
D.~R.~Stinson. The existence of Howell designs of odd side.  {\em J.\ Combin. Theory Ser. A} {\bf 32} (1982), 53--65.

\bibitem{Todorov}
D.~T.~Todorov.  Four mutually orthogonal Latin squares of order 14.  {\em J.\ Combin.\ Des.} {\bf 20} (2012), 363--367.

\bibitem{Vanstone 80}
S.~A.~Vanstone. Doubly resolvable designs. {\em Discrete Math.} {\bf 29} (1980), 77--86.

\bibitem{Vanstone 82} 
S.~A.~Vanstone. On mutually orthogonal resolutions and near resolutions. {\em Ann.\ of Discrete Mathematics} {\bf 15} (1982), 357--369.

\bibitem{WangDu}
C.~Wang and B.~Du. Existence of generalized Howell designs of side $n+1$ and order $3n$. {\em Utilitas Math.} {\bf 80} (2009), 143--159.

\bibitem{YanYin}
J. Yan and J. Yin. Constructions of optimal GDRP$(n,\lambda;v)$s of type  $\lambda^{1} \mu^{m-1}$.  {\em Discrete Appl. Math.} {\bf 156} (2008), 2666--2678.
\end{thebibliography}
\end{document}